\documentclass[11pt,english]{article}

\usepackage[T1]{fontenc}
\usepackage[latin9]{inputenc}
\usepackage{geometry}
\usepackage{color}
\usepackage{amsmath}
\usepackage{amsthm}
\usepackage{amssymb}
\usepackage{stmaryrd}
\usepackage{graphicx}
\usepackage{esint}
\usepackage{mathrsfs}  
\makeatletter

\numberwithin{equation}{section}
\numberwithin{figure}{section}
\numberwithin{table}{section}

\theoremstyle{plain}

\newtheorem{thm}{\protect\theoremname}[section]
  \theoremstyle{remark}
  \newtheorem{rem}[thm]{\protect\remarkname}
  \theoremstyle{plain}
  \newtheorem{cor}[thm]{\protect\corollaryname}
  \theoremstyle{plain}
  \newtheorem{lem}[thm]{\protect\lemmaname}
  \theoremstyle{definition}
  \newtheorem{defn}[thm]{\protect\definitionname}
  \theoremstyle{remark}
  \newtheorem{claim}[thm]{\protect\claimname}
  \theoremstyle{plain}
  \newtheorem{prop}[thm]{\protect\propositionname}

\usepackage{bbm}
\newcommand{\ind}{{\mathbbm{1}}}

\makeatother

\usepackage{babel}
  \providecommand{\claimname}{Claim}
  \providecommand{\corollaryname}{Corollary}
  \providecommand{\definitionname}{Definition}
  \providecommand{\lemmaname}{Lemma}
  \providecommand{\propositionname}{Proposition}
  \providecommand{\remarkname}{Remark}
\providecommand{\theoremname}{Theorem}

\begin{document}

\title{Eigenvalue confinement and spectral gap for random simplicial complexes}

\author{Antti Knowles\thanks{Partially supported by Swiss National Science Foundation grant 144662.} \and Ron Rosenthal\thanks{Partially supported by an ETH fellowship.}}

\maketitle

\begin{abstract}
	We consider the adjacency operator of the Linial-Meshulam model for random simplicial complexes on $n$ vertices, where each $d$-cell is added independently with probability $p$ to the complete $(d-1)$-skeleton. Under the assumption $np(1-p) \gg \log^4 n$, we prove that the spectral gap between the $\binom{n-1}{d}$ smallest eigenvalues and the remaining $\binom{n-1}{d-1}$ eigenvalues is $np - 2\sqrt{dnp(1-p)} \, (1 + o(1))$ with high probability. This estimate follows from a more general result on eigenvalue confinement. In addition, we prove that the global distribution of the eigenvalues is asymptotically given by the semicircle law. The main ingredient of the proof is a F\"uredi-Koml\'os-type argument  for random simplicial complexes, which may be regarded as sparse random matrix models with dependent entries. 
\end{abstract}


\section{Introduction \label{sec:Introduction}}
The Erd\H{o}s-R\'enyi graph \cite{ER59,ER61} $G(n,p)$ is a random graph on $n$ vertices, where each edge is added independently with probability $p$. The spectrum of its adjacency matrix has been extensively studied \cite{FK81,FKS89,FO05,CO07,HKP12,EKYY12,EKYY13}. 
Generally, the spectrum of the adjacency matrix of a graph encodes many important properties of the graph, in particular relating to its connectivity and expansion properties \cite{AM85,Alo86,Ni91,Ch97,HLW06,KS06,CRS10}. From the point of view or random matrix theory, the adjacency matrix of $G(n,p)$ is a symmetric sparse random matrix with independent upper-triangular entries.

In this paper we study the spectra of random simplicial complexes. We consider the \emph{Linial-Meshulam model} \cite{LM06}, which is high-dimensional generalization of the Erd\H{o}s-R\'enyi model. Given $n,d\in\mathbb{N}$ and $p\in[0,1]$, the Linial-Meshulam model $X \equiv X(d,n,p)$ is a random $d$-dimensional complex on $n$ vertices with a complete $(d-1)$-skeleton in which each $d$-cell is added independently with probability $p$. For $d = 1$, the Linial-Meshulam model reduces to the Erd\H{o}s-R\'enyi model $G(n,p)$. Following its introduction in \cite{LM06}, the Linial-Meshulam has been extensively studied in \cite{MW09,Ko10,BHK11,Wa11,CCFK12,HKP12,HJ13Th,HKP13,ALLM13,GW14,LP14}. The notion of adjacency matrix has a natural extension to simplicial complexes, whereby the adjacency operator of a complex $X$ is a self-adjoint operator that encodes the information whether two $(d-1)$-cells belong to a common $d$-cell or not.

As for graphs, the spectrum of the adjacency operator, in particular its spectral gap, determines a notion of spectral expansion. There has recently been considerable interest in \emph{high-dimensional expanders}, namely analogs of expander graphs in the context of
general simplicial complexes. Unlike the graph case $d=1$, where various different notions of expansion are closely related, in the high-dimensional case $d>1$ the several notions of expansion that have been proposed in the literature are in general far from being equivalent, and the relationship between them is still poorly understood. Notions of expansion for $d>1$ that have been proposed include the aforementioned spectral expansion \cite{Gar73,GW14,GP14,Op14}, combinatorial expansion \cite{PRT12,Par13,Go13,GS14,CMRT14},
geometric and topological expansion \cite{Gr10,FGLNP10,MW11,DKW15,Ev15}, $\mathbb{F}_{2}$-coboundary expansion \cite{LM06,MW09,Gr10,DK12,SKM12,GW14,LM13} and Ramanujan complexes \cite{CSZ03,Li04,LSV05,GP14,EGL14,KKL14}. In addition, the adjacency matrix can be  interpreted as a generator of a stochastic process which is a high-dimensional analog of simple random walks, see \cite{PR12,MS13,Ro14} for more details. A recent survey can be found in \cite{Lu13}.

From the point of view of random matrix theory, the adjacency matrix of the Linial-Meshulam model is a sparse self-adjoint random matrix. The entries of this matrix are independent (up to the self-adjointness constraint) if and only if $d = 1$. Indeed, the independent random variables are associated with \emph{$d$-cells}, while entries of the adjacency matrix are associated with \emph{pairs of $(d - 1)$-cells}; these two notions coincide if and only if $d = 1$. The algebraic relationship between the independent random variables and the matrix entries is governed by the simplicial structure.

We now give an informal summary of our results. Throughout the following we use the abbreviation 
\begin{equation}
	q := p(1-p).\label{eq:defn_q}
\end{equation}
Under the assumption $nq \gg \log^4 n$ and fixed $d$, we prove that the $N := \binom{n}{d}=\binom{n-1}{d}+\binom{n-1}{d-1}$ eigenvalues of the adjacency operator are with high probability confined to two separate intervals: denoting by $\lambda_1 \leq \lambda_2 \leq \cdots \leq \lambda_N$ the eigenvalues of the adjacency operator, we have with high probability
\begin{equation} \label{result_sketch}
	\lambda_i \;\in\;
	\begin{cases}
		\sqrt{dnq} \, \bigl[-2 - o(1), 2  + o(1)\bigr] & \text{if } i \leq \binom{n-1}{d}\\
		np + [-7d,7d] & \text{if } i > \binom{n-1}{d}\,
	\end{cases}.
\end{equation}
The first estimate is optimal, while the second one is not (in fact the constant $7$ may be easily improved). See Figure \ref{fig:eigenvalues_illustration} for an illustration. As an immediate corollary, we obtain the spectral gap
\begin{equation} \label{sketch_gap}
	\lambda_{\binom{n-1}{d} + 1} - \lambda_{\binom{n-1}{d}} \;=\; np - 2\sqrt{dnq} \, (1 + o(1))\,.
\end{equation}
We refer to Theorem \ref{thm:location_of_eigenvalues} and Corollary \ref{cor:con_gap} below for the precise statements.

\begin{figure}[h]
\centering{}
\includegraphics[scale=0.6]{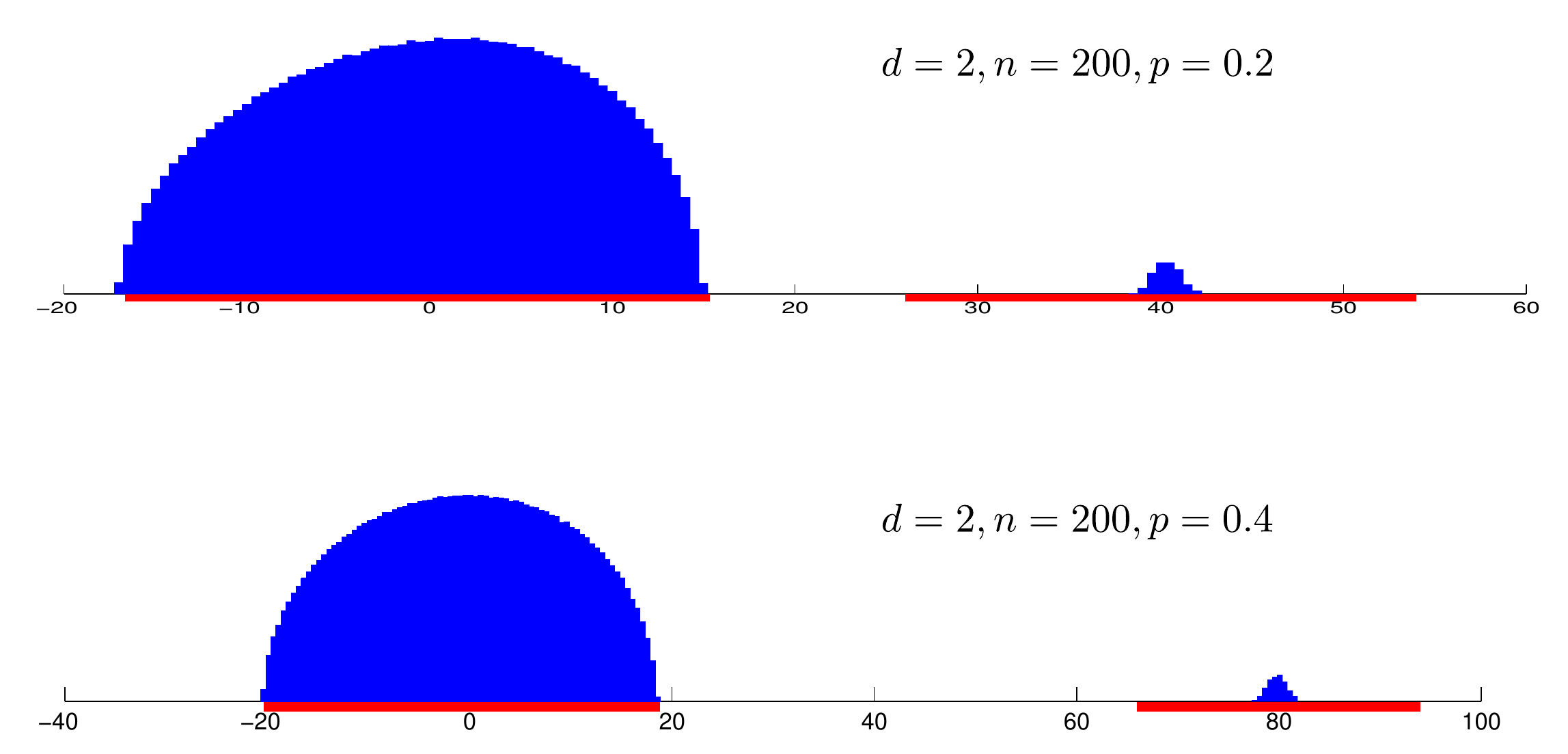}
\caption{A histogram of the eigenvalue distribution of $A$ for the values $(d,n,p)=(2,200,0.2)$ (top) and $(d,n,p)=(2,200,0.4)$ (bottom). The two intervals on the right-hand side of \eqref{result_sketch} are indicated using red regions.}
\label{fig:eigenvalues_illustration}
\end{figure}

Previously, a related confinement result for the eigenvalues of the Linial-Meshulam model was established in \cite[Theorem 2]{GW14}, where the authors assume that $nq \gg \log n$ and establish \eqref{result_sketch} with the larger intervals $\sqrt{nq} [-C,C]$ for $i \leq \binom{n-1}{d}$ and $np + \sqrt{nq} \, [-C,C]$ for $i > \binom{n-1}{d}$, where $C$ is some positive constant. This implies a spectral gap \eqref{sketch_gap} equal to $np + O(\sqrt{nq})$. Hence, at the cost of the stronger assumption $nq \gg \log^4 n$ instead of $nq \gg \log n$, we improve the confinement of \cite{GW14} by obtaining the optimal constant for the size of the interval in the first case of \eqref{result_sketch}, and by improving the size of the interval in the second case of \eqref{result_sketch} by a factor $(nq)^{-1/2}$. Together, these improvements allow us to identify the subleading term in the spectral gap \eqref{sketch_gap}.

Finally, we prove that when $nq \to \infty$ the empirical spectral measure is asymptotically given by Wigner's semicircle law \cite{Wi58} on the interval $\sqrt{dnq}\,[-2,2]$. In particular, this implies the optimality of the first bound of \eqref{result_sketch}. This result extends the well-known semicircle law for $G(n,p)$ to high-dimensional simplicial complexes.

We conclude this section with a few words about the proof. The main ingredient of the proof is a F\"uredi-Koml\'os-type argument for random simplicial complexes. The F\"uredi-Koml\'os argument involves estimating the expectation of the trace of a very high power of the centered adjacency matrix, by encoding the many resulting terms using walks on graphs. A fundamental observation in the original F\"uredi-Koml\'os argument is that, owing to the vanishing expectation of the entries of the centered adjacency matrix, each edge of the graph must be crossed at least twice, leading to a reduction in the number of admissible walks. For $d > 1$, this is no longer true because of the dependencies among the matrix entries. Obtaining sharp enough upper bounds on the number of admissible walks represents the main work in our proof.  As it turns out, for $d > 1$ the mechanism behind the reduction in the number of admissible walks is more subtle, and, unlike for $d = 1$, nonlocal. We refer to Section \ref{sec:Bounding-the-norm-of-H} for a more detailed discussion on how to estimate the number of admissible paths. Finally, in Section \ref{sec:location_of_eigenvalues}, in order to prove the second case of \eqref{result_sketch}, we need to modify the argument described above to obtain smaller bounds, by a factor $(nq)^{-1/2}$, for the restriction of the adjacency matrix to an explicit $\binom{n-1}{d-1}$-dimensional subspace (denoted by $\mathrm{im}\,P$ in Section \ref{sec:location_of_eigenvalues}). This estimate is obtained by a twist of the argument developed in Section \ref{sec:Bounding-the-norm-of-H}. Our main results then follow easily from these estimates, combined with eigenvalue interlacing bounds and second order perturbation theory (see Sections \ref{sec:pf_thm26} and \ref{sec:location_of_eigenvalues}).


\section{Results}

Let $X$ be a finite simplicial complex with vertex set $V$ of size $n$. This means that $X$ is a finite collection of subsets of $V$, called \emph{cell}s, which is closed under taking subsets, i.e., if $\tau\in X$ and $\sigma\subseteq\tau$, then $\sigma\in X$. The \emph{dimension }of a cell $\sigma$ is $|\sigma|-1$, and $X^{j}$ denotes the set of cells of dimension $j$, which we call \emph{$j$-cells}. The dimension of $X$, which we denote by $d$, is the maximal dimension of a cell in it. We assume that the complex has a \emph{complete $(d-1)$-skeleton}, by which we mean that $X$ contains all subsets of $V$ of size at
most $d$. 

For $j\geq 1$, every $j$-cell $\sigma=\{ \sigma^{0},\ldots,\sigma^{j}\} \in X^j$ has two possible orientations, corresponding to the possible orderings of its vertices, up to an even permutation. We denote an oriented cell by square brackets, and a flip of orientation by an overline. For example, one orientation of $\sigma=\{x,y,z\} $ is $[x,y,z]=[y,z,x]=[z,x,y]$. The other orientation of $\sigma$ is $\overline{[x,y,z]}=[y,x,z]=[x,z,y]=[z,y,x]$. We denote by $X_{\pm}^{j}$ the set of oriented $j$-cells (so that $|X_{\pm}^{j}|=2|X^{j}|$ for $j\geq1$). Moreover, we set $X_{\pm}^{0}=X^{0}=V$.

For $j \geq 1$, the space of \emph{$j$-forms on $X$}, denoted by $\Omega^{j}(X)$, is the vector space of skew-symmetric functions on oriented $j$-cells:
\[
	\Omega^{j}\equiv\Omega^{j}(X):= \left\{ f:X_{\pm}^{j}\rightarrow\mathbb{R}\,\middle|\,f(\overline{\sigma})=-f(\sigma)\;\forall\sigma\in X_{\pm}^{j}\right\} .
\]
In particular, 
$\Omega^{1}$ is the space of flows on edges. We endow each $\Omega^{j}$
with the inner product
\[
	\langle f,g \rangle =\sum_{\sigma\in X^{j}}f(\sigma)g(\sigma).
\]
Note that $f(\sigma)g(\sigma)$ is well-defined since its value is independent of the choice of orientation of the $j$-cell $\sigma$. 

Next, define the boundary $\partial \sigma$ of the $(j+1)$-cell $\sigma=\{ \sigma^{0},\ldots,\sigma^{j+1}\} \in X^{j+1}$ as the set of $j$-cells $\{ \sigma^{0},\ldots,\sigma^{i-1},\sigma^{i+1},\ldots,\sigma^{j}\}$ for $0\leq i\leq j+1$. An oriented $(j+1)$-cell $[\sigma^{0},\ldots,\sigma^{j+1}]\in X_{\pm}^{j+1}$ induces orientations on the $j$-cells in its boundary, as follows: the cell $\{ \sigma^{0},\ldots,\sigma^{i-1},\sigma^{i+1},\ldots,\sigma^{j+1}\}$ is oriented as $(-1)^{i}[\sigma^{0},\ldots,\sigma^{i-1},\sigma^{i+1},\ldots,\sigma^{j+1}]$, where we use the notation  $(-1)\tau:=\overline{\tau}$.

The following neighboring relation for oriented cells was introduced in \cite{PR12}: for $\sigma,\sigma'\in X_{\pm}^{d-1}$ we denote $\sigma'\sim\sigma$ (or $\sigma\overset{X}{\sim}\sigma'$) if there exists an oriented $d$-cell $\tau\in X_{\pm}^{d}$ such that both $\sigma$ and $\overline{\sigma'}$ are in the boundary of $\tau$ as oriented cells (see Figure \ref{fig:An-oriented-edge} for an illustration in the case $d=2$).

\begin{figure}[h]
\centering{}\includegraphics{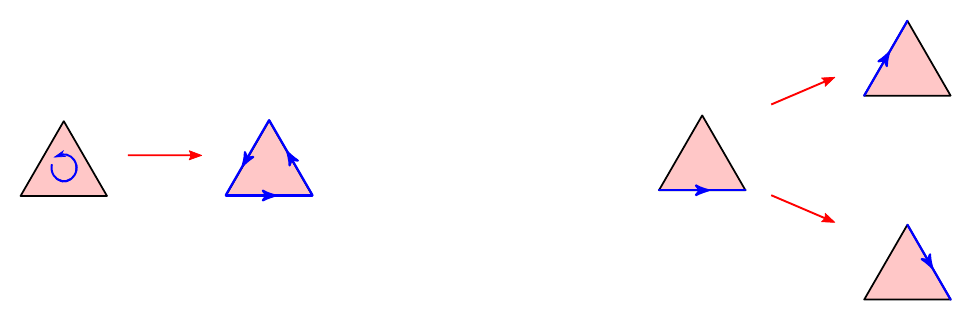}\caption{Left: an oriented 2-cell and the orientation it induces on its boundary. Right: an oriented 1-cell in a 2-cell together with its two oriented neighboring 1-cells.
\label{fig:An-oriented-edge}}
\end{figure}

The adjacency operator $A=A_{X}$ of a complex $X$ on the space $\Omega^{d-1}(X)$
is defined by 
\begin{equation}
	Af(\sigma):=\sum_{\sigma'\sim\sigma}f(\sigma'),\qquad\forall f\in\Omega^{d-1}(X),\,\sigma\in X_{\pm}^{d-1}.\label{eq:Adj_operator}
\end{equation}
This definition is a rather direct way to introduce the adjacency operator. An equivalent, and more conceptual, definition is $A=D-\Delta^{+}$, where $D$ is the degree operator
of $(d-1)$-cells (the \emph{degree} of a $j$-cell is the number of $(j+1)$-cells which contain it) and $\Delta^{+}$ is the upper Laplacian originating in the work of Eckmann \cite{Eck44}. More on the connection between $A,\Delta^{+}$ and the homology and cohomology of the complex can be found in \cite{Ha02,GW14}. See also \cite{PRT12,MS13,Ro14} for the connection to random walks on simplicial complexes.

We denote by $K \equiv K(d,n)$ the complete $d$-complex on the $n$ vertices in $V$ and by $\mathbb{A}:=A_{K}$ its adjacency operator. 

Coming back to the Linial-Meshulam model, for $n,d\in\mathbb{N}$ satisfying $n \geq d+1$, and $p=p(n)\in[0,1]$, the Linial-Meshulam complex $X=X(d,n,p)$ is a random $d$-dimensional complex on $n$ vertices with a complete $(d-1)$-skeleton in which each $d$-cell of $K$ is added to $X$ independently with probability $p$. This in particular implies that $X^{j}=K^{j}$ for every $0\leq j\leq d-1$ and thus also $\Omega^{j}(X)=\Omega^{j}(K)$ for $0\leq j\leq d-1$. We denote by $\mathbb{P}  \equiv \mathbb{P}_n $ and $\mathbb{E}\equiv \mathbb{E}_n $ the probability measure and expectation with respect to the law of $X$. Note that for $d=1$ this is exactly the Erd\H{o}s-Rényi model $G(n,p)$.

Throughout this paper we fix $d\geq2$. All of our results also hold for the case $d=1$ which was already extensively studied in the literature; see \cite{Wi58,FK81,FO05,CO07,HKP12} and the references therein. The adjacency operator of $X$, $A=A_{X}$, is a self adjoint operator on the $N \equiv |X^{d-1}|=\binom{n}{d}$-dimensional space $\Omega^{d-1}(X)$.

We now state our main results.
\begin{thm}[Eigenvalue confinement] \label{thm:main_result} \label{thm:location_of_eigenvalues} For every $d\geq2$ there exists a positive constant $C>0$  depending only on $d$ such that the following holds with probability at least $1-n^{-D}$ for all $D > 0$. 
\begin{enumerate}
	\item For every $\xi>0$, if $nq\geq\frac{C(1+D)^{4}}{\min\{ \xi^6,1\}}\log^{4}n$  then the $\binom{n-1}{d}$ smallest eigenvalues of $A$ lie in the interval $\sqrt{dnq}  \, [-2-\xi,2+\xi]$.  
	\item If $nq\geq C(1+D)^4\log^{4}n$ and $q\log^{6}n\leq\frac{1}{C(1+D)^6}$, then the $\binom{n-1}{d-1}$ largest eigenvalues of $A$ lie in the interval $np+[-7d,7d]$. 
\end{enumerate}
\end{thm}

\begin{rem} \label{rem:generalization}
	In Theorem \ref{thm:location_of_eigenvalues_strong_ver} below, we give a slightly more precise version of Theorem \ref{thm:main_result}. For instance, we can replace (2) with the following statement.
	\begin{itemize}
		\item[($\mathit{2}$')] If $nq\geq C(1+D)^4 \log^{4}n$ then the $\binom{n-1}{d-1}$ largest eigenvalues of $A$ lie in the interval
\begin{equation}
	np+\left(6d+O_D(\sqrt{q}\log^3n)\right) [-1,1].
\end{equation}
\end{itemize}
\end{rem}

From Theorem \ref{thm:main_result} and Remark \ref{rem:generalization}, we immediately get the following result.

\begin{cor}[Spectral gap]
	For every $d\geq2$ there exists a positive constant $C > 0$ depending only on $d$ such that for all $\xi > 0$ and $D > 0$ satisfying $nq\geq\frac{C(1+D)^4}{\min\{\xi^6,1\}}\log^{4}n$ we have
\begin{equation*}
	\lambda_{\binom{n-1}{d} + 1} - \lambda_{\binom{n-1}{d}} = np - 2 \sqrt{dnq} \, \big(1+ O_D(\xi)\big)
\end{equation*}
with probability at least $1-n^{-D}$.
\end{cor}

For the following results, where $n \to \infty$ in probability space, we couple all random complexes $(X(d,n,p))_{n \in \mathbb{N}}$ in the standard fashion. Namely, we work on the probability space generated by the family of i.i.d.\ random variables $(\chi_\tau)_\tau$ indexed by all $d$-cells $\tau$ on the infinite vertex set $\{1,2,3,\dots\}$, where $\chi_\tau$ is a Bernoulli-$p$ random variable. Thus, for any $n\geq1$ and a $d$-cell $\tau \subset \{1, \dots, n\}$, we set $\tau$ to belong to $X(d,n,p)$ if and only if $\chi_{\tau}=1$.

\begin{cor} \label{cor:con_gap}
	Fix $d\geq2$ and an open interval $I$ containing $[-2,2]$. If $\lim_{n\to\infty}\frac{nq}{\log^{4}n}=\infty$, then, almost surely, the $\binom{n-1}{d}$ smallest eigenvalues of $A$ lie in $\sqrt{dnq}\,I$ for all but finitely many values of $n$.
\end{cor}
An analogous result holds for the largest $\binom{n-1}{d-1}$ eigenvalues, whose precise statement we omit. 

Our final result is about the eigenvalue distribution of $A$. For a self-adjoint operator $B$ operating on an $N$-dimensional space, denote by $(\lambda_{i}(B))_{i=1}^{N}$ its
eigenvalues in non-decreasing order and by 
\begin{equation}
	L_{B}=\frac{1}{N}\sum_{i=1}^{N}\delta_{\lambda_{i}(B)}\label{eq:empirical_measure}
\end{equation}
the empirical spectral measure of $B$.

\begin{thm}
	\label{thm:Semi_circle_law} If $d\geq2$ and $\lim_{n\to\infty}nq=\infty$,
	then, almost surely, 
	\begin{equation}
		L_{(dnq)^{-1/2}A}\overset{\cal D}{\longrightarrow}\nu\,, \label{eq:semi_circle}
	\end{equation}
	where $\nu(\mathrm{d}x) := \frac{\sqrt{(4-x^{2})_+}}{2\pi} \, \mathrm{d}x$ is the semicircle 	distribution.
\end{thm}

\begin{rem}
\label{rem:optimal_constant}
	It follows from Theorem \ref{thm:Semi_circle_law} that the constant $2$ in the interval $\sqrt{dnq} \, \bigl[-2-\xi,2+\xi\bigr]$ from Theorem \ref{thm:location_of_eigenvalues}(1) is optimal. In fact, it is an easy corollary of Theorems \ref{thm:main_result} and \ref{thm:Semi_circle_law} that if $\lim_{n\to\infty}\frac{nq}{\log^{4}n}=\infty$ then $\lim_{n\to\infty}(dnq)^{-1/2}\lambda_{\binom{n-1}{d}} = 2$ and $\lim_{n\to\infty} (dnq)^{-1/2} \lambda_1 = -2$ almost surely. See also Remark \ref{rem:conv_eig} below.
\end{rem}

\paragraph{Conventions.}
	We use $C$ to denote a generic large positive constant, which may depend on some fixed parameters and whose value may change from one expression to the next.  If $C$ depends on some parameter $k$, we sometimes emphasize this dependence by writing $C_k$ instead of $C$. Moreover, for $f,g:\mathbb{N}\to\mathbb{R}$ we write $f(n)=O_{k}(g(n))$ to mean $f(n) \leq C_k g(n)$ for all $n\in\mathbb{N}$. Finally, we abbreviate $\llbracket n\rrbracket := \{1, \dots, n\}$. The letters $d,i,j,k,l,m,n,s,N$ are always used to denote an element
in $\mathbb{N}=\{ 0,1,2,\ldots\}$.

From now on, we consistently use $\sigma$ for (oriented or non-oriented) $(d-1)$-cells, and $\tau$ for (oriented or non-oriented) $d$-cells.


\section{Semicircle law for $A-\mathbb{E}[A]$}

In this section we prove the semicircle law for the centered and normalized adjacency matrix  
\begin{equation}
	H:=\frac{1}{\sqrt{nq}}(A-\mathbb{E}[A]).\label{eq:The_matrix_H}
\end{equation}

More precisely, we prove that the empirical spectral measure $L_H$ converges to the rescaled distribution $\nu_d(\mathrm{d}x) := \frac{\sqrt{(4d-x^{2})_+}}{2\pi d} \, \mathrm{d}x$. 

\begin{thm}\label{thm:Semicircle_for_H} 
	Fix $d\geq2$. If $\lim_{n\to\infty}nq=\infty$, then as $n \to \infty$ almost surely $L_{H}\overset{\cal D}{\longrightarrow}\nu_{d}$.
\end{thm}

Using Theorem \ref{thm:Semicircle_for_H}, we shall conclude the proof of Theorem \ref{thm:Semi_circle_law} in Section \ref{sec:pf_thm26} below. The proof of Theorem \ref{thm:Semicircle_for_H} is based on the moment method (see e.g.\ the presentation of \cite[Section 2.1]{AGZ10} for the classical case of independent matrix entries), and is the subject of the rest of this section. Along the proof, we also record several definitions and notions that will be used for the proof of Theorem \ref{thm:main_result} in Sections \ref{sec:Bounding-the-norm-of-H}--\ref{sec:location_of_eigenvalues}.

By a standard truncation argument and the Borel-Cantelli lemma, Theorem \ref{thm:Semicircle_for_H} follows from the two next lemmas (see \cite[Section 2.1.2]{AGZ10} for details). 

\begin{lem} \label{lem:Semicircle_moment_lemma} 
	If $\lim_{n\to\infty}nq=\infty$, then for every fixed $k\in\mathbb{N}$
	\[
		\lim_{n\to\infty}\mathbb{E}\left[\int_{\mathbb{R}}x^{k}L_{H}(\mathrm{d}x)\right]=	\int_{\mathbb{R}}x^{k}\nu_{d}(\mathrm{d}x)=
		\begin{cases}
			0 & \text{if }k\,\mbox{is odd}\\
			d^{k/2}\mathcal{C}_{k/2} & \text{if }k\,\mbox{is even}
		\end{cases},
	\]
	where $\mathcal{C}_{k}:=\frac{1}{k+1}\binom{2k}{k}$ is the $k$-th Catalan number. 
\end{lem}

\begin{lem}\label{lem:Semicircle_variance_lemma} 
	For every fixed $k\in\mathbb{N}$ we have
	\begin{equation}
		\mathrm{Var}\left(\int_{\mathbb{R}}x^{k}L_{H}(\mathrm{d}x)\right)\leq O_{k}\left(\frac{1}{n^{d}(nq)}\right).\label{eq:variance_estimation}
	\end{equation}
\end{lem}

The rest of this section is devoted to the proof of Lemma \ref{lem:Semicircle_moment_lemma}. The proof of Lemma \ref{lem:Semicircle_variance_lemma} is a standard adaptation of the ideas of the proof of Lemma \ref{lem:Semicircle_moment_lemma}, and we omit its details. We begin with definitions that we use throughout the remainder of the paper.

First, we need an explicit matrix representation of the adjacency operator. We fix an arbitrary choice of orientation of the $(d-1)$-cells $K_{+}^{d-1}\subset K_{\pm}^{d-1}$, which in turn induces a choice of orientation on the $(d-1)$-cells of $X$ since $X^{d-1}=K^{d-1}$. Note that there is a natural bijection between $K_{+}^{d-1}$ and $K^{d-1}$, and hence also between $X_{+}^{d-1}$ and $X^{d-1}$. From now on, by a slight abuse of notation, using this bijection we often write $\sigma_1 \cup \sigma_2$ and $\sigma_1 \cap \sigma_2$ for oriented $(d-1)$-cells $\sigma_1,\sigma_2$ to denote the union and intersection of the corresponding unoriented cells.

The orientation $K_{+}^{d-1}\subset K_{\pm}^{d-1}$ gives rise to an associated orthonormal basis $(\ind_{\sigma})_{\sigma \in X_+^{d-1}}$ of $\Omega^{d-1}$,  defined by
\[
	\ind_{\sigma}(\sigma')=
	\begin{cases}
		1 & \quad \text{if } \sigma'=\sigma\\
		-1 & \quad \text{if } \sigma'=\overline{\sigma}\\
		0 & \quad \mbox{otherwise}
	\end{cases}.
\]
Then the adjacency matrix is an $X_+^{d-1} \times X_+^{d-1}$ matrix $(A_{\sigma \sigma'})$ with entries $A_{\sigma \sigma'} := \langle \ind_\sigma , A \ind_{\sigma'} \rangle$. Explicitly,
\begin{equation}
	A_{\sigma,\sigma'}=
	\begin{cases}
		1 & \quad \text{if }\sigma\overset{_{X}}{\sim}\sigma'\\
		-1 & \quad \text{if }\sigma\overset{_{X}}{\sim}\overline{\sigma'}\\
		0 & \quad \mbox{otherwise}
	\end{cases}.
	\label{eq:adj_matrix}
\end{equation}
In particular,
\begin{equation}
	\mathbb{E}[A]=p\mathbb{A},\label{eq:expectation_of_adj_matrix}
\end{equation}
where we recall that $\mathbb{A}$ is the adjacency operator of the complete $d$-complex $K$.

Next, we introduce the basic definitions underlying the proof of Lemma \ref{lem:Semicircle_moment_lemma}. They are illustrated in Figure \ref{fig:defn_illustration}.

\begin{defn}[Words] \label{def:letter_words_and_equivalence} 
	An $(n,d)$-letter $\sigma$ (or shortly a \emph{letter}) is an element of 	$X_{+}^{d-1}$. An $(n,d)$-word $w$ (or shortly a \emph{word}) is a finite sequence $\sigma_{1}\ldots\sigma_{k}$ of letters at least one letter long such that $\sigma_{i}\cup\sigma_{i+1}$ is a $d$-cell in $K$ for every $1\leq i\leq k-1$. The \emph{length} of the word $\sigma_{1}\ldots\sigma_{k}$ is defined to be $k$. A word is called \emph{closed} if its first and last letters are the same. Two words $w=\sigma_{1}\ldots\sigma_{k}$ and $w'=\sigma'_{1}\ldots\sigma'_{k}$ are called \emph{equivalent} if there exists a permutation $\pi$ on $X^{0}=V$ such that $\pi(\sigma_{i})=\sigma'_{i}$ for every $1\leq i\leq k$, where for $\sigma=[\sigma^{0},\ldots,\sigma^{d-1}]\in X_{\pm}^{d-1}$
we write $\pi(\sigma)=[\pi(\sigma^{0}),\ldots,\pi(\sigma^{d-1})]$. 
\end{defn}

\begin{defn}[Support of a word]\label{def:length_supp_and_d-supp_for_words}
	For a word $w=\sigma_{1}\ldots\sigma_{k}$ we define its \emph{support} by $\mathrm{supp}_{0}(w)=\sigma_{1}\cup\sigma_{2}\cup\ldots\cup\sigma_{k}$ and its \emph{$d$-cell support} by $\mathrm{supp}_{d}(w)=\{ \sigma_{i}\cup\sigma_{i+1}\,:\,1\leq i\leq k-1\}$.
\end{defn}

\begin{defn}[Graph of a word]\label{def:graph_of_a_word} 
	Given a word $w=\sigma_{1}\ldots\sigma_{k}$ we define $G_{w}=(V_{w},E_{w})$ to be the graph with vertex set $V_{w}=\{ \sigma_{i}\,:\,1\leq i\leq k\}$ and edge set $E_{w}=\{\{ \sigma_{i},\sigma_{i+1}\} \,:\,1\leq i\leq k-1\}$. To avoid confusion, from this point on we consistently use the words vertex and edge to describe elements of $V_{w}$ and $E_{w}$ respectively; elements in $X^{0}$ and $X^{1}$ are always referred to as $0$-cells and $1$-cells. 

The graph $G_{w}$ comes with a path, given by the word $w$, that goes through all of its vertices and edges. We call each step along the path, i.e.\ $\sigma_{i}\sigma_{i+1}$ for some $1\leq i\leq k-1$, a \emph{crossing} of the edge $\{ \sigma_{i},\sigma_{i+1}\}$ and a \emph{crossing} of the $d$-cell $\sigma_{i}\cup\sigma_{i+1}$. The path gives a natural ordering of the vertices and edges of $G_{w}$, and of the $d$-cells in $\mathrm{supp}_{d}(w)$ according to the order of their first crossing along $w$. Finally, for an edge $e\in E_{w}$ define $N_{w}(e)$ to be the number of times the edge $e$ is crossed along the path generated by $w$ in the graph $G_{w}$. Given an edge $e\in E_{w}$ and $1\leq i\leq N_{w}(e)$, the \emph{$i$-th crossing time of the edge $e$} is given by the unique $1\leq j\leq k-1$ such that $\sigma_{j}\sigma_{j+1}$ is the $i$-th crossing of the edge $e$ along the path $w$. 
\end{defn}

The edges of a graph $G_{w}$, associated with a word $w=\sigma_{1}\ldots\sigma_{k}$, can be divided into different classes according to the $d$-cell generated by the two $(d-1)$-cells in its endpoints. This is done as follows.

\begin{defn}\label{def:Counting_crossing_in_a_graph_word} 
	For a $d$-cell $\tau$ let $E_{w}(\tau)=\{ \{ \sigma,\sigma'\} \in E_{w}\,:\,\sigma\cup\sigma'=\tau\}$ and define (with a slight abuse of notation) 
\[
	N_{w}(\tau)=\sum_{e\in E_{w}(\tau)}N_{w}(e)
\]
to be the total number of times the $d$-cell $\tau$ is crossed along the path generated by the word $w$. As for edges of the graph, given $\tau\in X^{d}$ and $1\leq i\leq N_{w}(\tau)$ we define the $i$-th crossing time of a $d$-cell $\tau$ to be the unique $1\leq j\leq k-1$ such that $\sigma_{j}\sigma_{j+1}$ is the $i$-th crossing of the $d$-cell $\tau$ along the path $w$. 

For $\sigma,\sigma'\in X_{+}^{d-1}$ the condition $\sigma\cup\sigma'\in K^{d}$ can mean one of two things, either $\sigma\overset{_{K}}{\sim}\sigma'$ or $\sigma\overset{_{K}}{\sim}\overline{\sigma'}$. For future use we denote the set of non-neighboring edges whose union is $\tau$ by 
\[
	\widehat{E}_{w}(\tau)=\big\{ \{ \sigma,\sigma'\} \in E_{w}(\tau)\,:\,\sigma\overset{_{K}}{\sim}\overline{\sigma'}\big\},
\]
so that $E_{w}(\tau)\setminus\widehat{E}_{w}(\tau)=\{ \{ \sigma,\sigma'\} \in E_{w}(\tau)\,:\,\sigma\overset{_{K}}{\sim}\sigma'\} $. 
\end{defn}

\begin{figure}[h]
	\centering{}\includegraphics[scale=0.5]{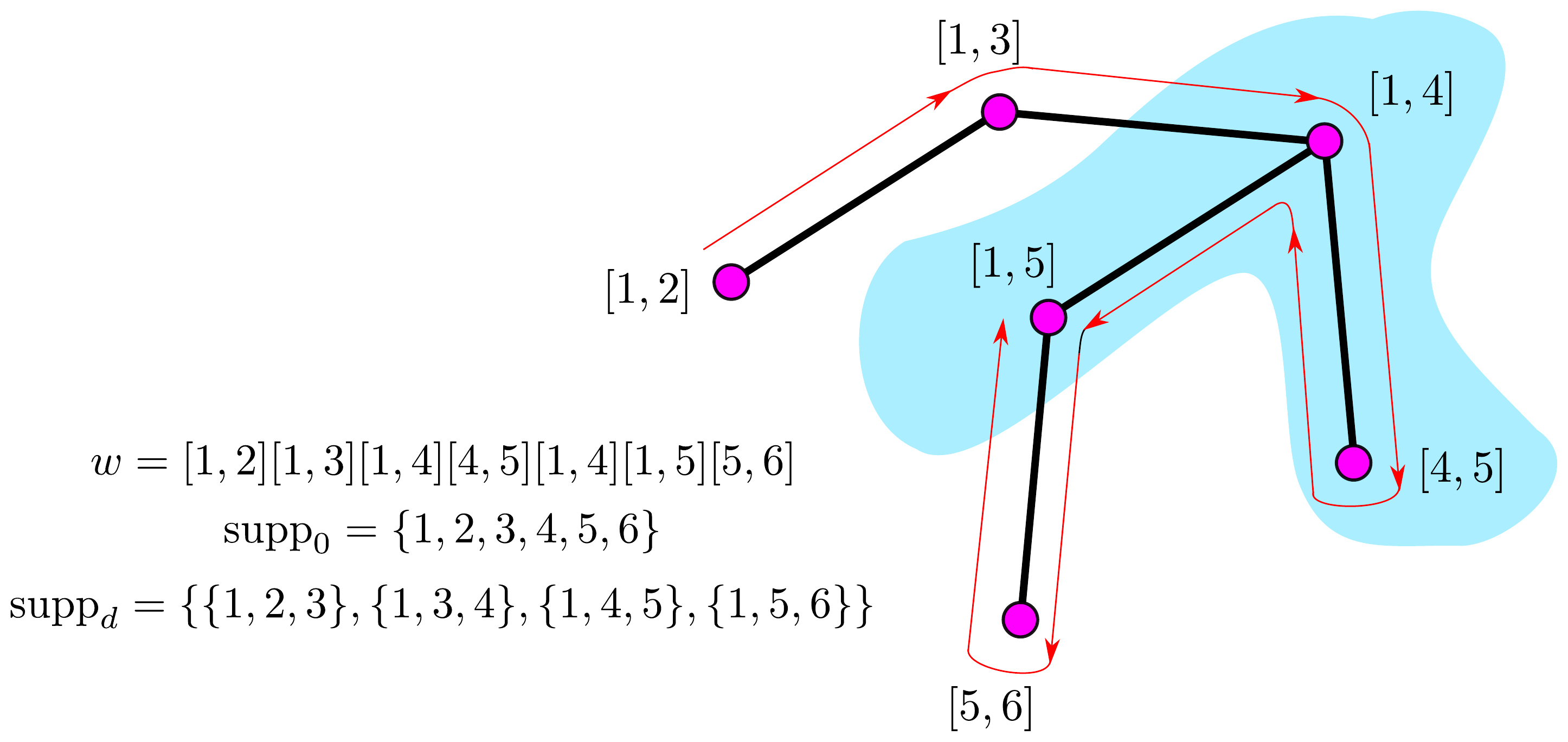}	
	\caption{\label{fig:defn_illustration} We illustrate Definitions \ref{def:letter_words_and_equivalence}-\ref{def:Counting_crossing_in_a_graph_word}
for the word $w=[1,2][1,3][1,4][4,5][1,4][1,5][5,6]$ (note that $w$ is indeed a word since the union of any two consecutive $1$-cells is a $2$-cell). The $0$-cell support $\mathrm{supp}_{0}(w)$ and $d$-cell support $\mathrm{supp}_{d}(w)$ (for $d=2$) of the word $w$ are given in the figure and the associated graph is drawn. The path induced by $w$ on the graph $G_{w}$ is illustrated by the red path. Note that the edges $\{[1,4],[1,5]\} $  and $\{[1,4],[4,5]\}$ (see blue area) belong to the same $2$-cell $\{ 1,4,5\} $. We therefore
have that $N_{w}(\{ 1,4,5\})=N_{w}(\{ [1,4],[1,5]\})+N_{w}(\{[1,4],[4,5]\})=2+1=3$.
As for the non-neighboring edges, we have for example $\widehat{E}_{w}(\{ 1,4,5\})=\{\{ [1,4],[4,5]\}\} $. The first crossing time of the edge $\{[1,4],[4,5]\}$ is $3$ and the third crossing time of the $d$-cell $\{ 1,4,5\} $ is $5$.}
\end{figure}

We remark that for $d=1$, there is no distinction between $N_w(\tau)$ and $N_w(e)$, since in that case the edges of the graph are naturally associated with $1$-cells.

Abbreviate $B=A-\mathbb{E}[A]=A-p\mathbb{A}$. For a fixed $k\in\mathbb{N}$, we wish to understand the limiting behavior of $\mathbb{E}\left[\int_{\mathbb{R}}x^{k}L_{H}\left(\mathrm{d}x\right)\right]$. Since for $k=0$ we have $\mathbb{E}\left[\int_{\mathbb{R}}L_{H}\left(\mathrm{d}x\right)\right]=1$ we will assume that $k\geq 1$. The starting point of the moment method is the identity
\begin{equation} \label{eq:semicircle_for_H_1} 
	\mathbb{E}\left[\int_{\mathbb{R}}x^{k}L_{H}(\mathrm{d}x)\right] =\frac{1}{N(nq)^{k/2}}\sum_{\sigma_{1},\ldots,\sigma_{k}\in X_{+}^{d-1}}\mathbb{E}[B_{\sigma_{1}\sigma_{2}}B_{\sigma_{2}\sigma_{3}}\ldots B_{\sigma_{k-1}\sigma_{k}}B_{\sigma_{k}\sigma_{1}}],
\end{equation}
which follows easily from \eqref{eq:empirical_measure}.

Note that each term in the sum can be associated with a string of letters
$\sigma_{1}\sigma_{2}\ldots\sigma_{k}\sigma_{1}$. Since $B_{\sigma,\sigma'}=0$
whenever $\sigma\cup\sigma'\notin X^{d}$ it follows that we can restrict
the sum in (\ref{eq:semicircle_for_H_1}) to the case where $\sigma_{i}\cup\sigma_{i+1}\in X^{d}$ for $1\leq i\leq k-1$ and $\sigma_{k}\cup\sigma_{1}\in X^{d}$. Consequently,
the list of letters in the sum can be restricted to the set of closed
words of length $k+1$. Using the independence structure of $A$ for
different $d$-cells and the definition of $N_{w}$ we then have 
\begin{equation}
	\mathbb{E}\left[\int_{\mathbb{R}}x^{k}L_{H}(\mathrm{d}x)\right]=
	\sum_{\substack{w\,\,\footnotesize{\mbox{a closed word}}\\
\footnotesize{\mbox{of length }}k+1}}
\frac{1}{N(nq)^{k/2}}\prod_{\tau\in X^{d}}\mathbb{E}\left[(\chi-p)^{N_{w}(\tau)}\right]\mathrm{sgn}(w,\tau),\label{eq:semicircle_for_H_2}
\end{equation}
where $\chi$ is a Bernoulli random variable with parameter $p$ and 
\begin{equation}
\mathrm{sgn}(w,\tau)=(-1)^{\sum_{e\in\widehat{E}_{w}(\tau)}N_{w}(e)}\label{eq:sign_function}
\end{equation}
is a sign that depends on the parity of the number of crossings between non-neighboring $(d-1)$-cells in $w$. We write $T(w)$ for the contribution of the word $w$ to the sum in (\ref{eq:semicircle_for_H_2}), that is 

\begin{equation}
	T(w)=\frac{1}{N(nq)^{k/2}}\prod_{\tau\in X^{d}}\mathbb{E}\left[(\chi-p)^{N_{w}(\tau)}\right]\mathrm{sgn}(w,\tau).\label{eq:semicircle_for_H_4}
\end{equation}
Since $\mathbb{E}[\chi-p]=0$, it follows that $T(w)=0$ unless $N_{w}(\tau)\neq 1$ for every $\tau\in X^{d}$. We can thus restrict the sum in (\ref{eq:semicircle_for_H_2}) to words such that $N_{w}(\tau)\neq 1$ for every $\tau\in X^{d}$.

Recalling the definition of the equivalence relation for words, see Definition \ref{def:letter_words_and_equivalence}, we denote by $[w]$ the equivalence class of $w$ and observe that, because the term $T(w)$ is invariant under permutations of $X^{0}$, we have $T(w)=T(w')$ for every pair of words $w,w'$ such that $[w]=[w']$.

Using the last observation and (\ref{eq:semicircle_for_H_4}) we can rewrite (\ref{eq:semicircle_for_H_2}) as 
\begin{equation}
	\mathbb{E}\left[\int_{\mathbb{R}}x^{k}L_{H}\left(\mathrm{d}x\right)\right]=\sum_{w}|[w]| \, T(w),
	\label{eq:semicircle_for_H_5}
\end{equation}
where the sum is over a set of representatives for the equivalence classes of closed words of length $k+1$ such that $N_{w}(\tau)\neq 1$ for every $\tau\in X^{d}$. 

Next, we distinguish between different equivalence classes according to the number of $0$-cells in their support. Denote by $\mathcal{W}_{s}^{k}=\mathcal{W}_{s}^{k}(n,d)$ a set of representatives for the equivalence classes of closed words of length $k+1$ such that $N_{w}(\tau)\neq 1$ for every $\tau\in X^{d}$ and $|\mathrm{supp}_{0}(w)|=s$. Note that this parameter is independent of the choice of the representative in the equivalence class. Using the fact that $\sum_{\tau\in X^{d}}N_{w}(\tau)=k$ and recalling that we only consider words such that $N_{w}(\tau)\neq 1$ for every $\tau\in X^{d}$, it follows that the number of $d$-cells such that $N_{w}(\tau)>0$ is bounded by $\left\lfloor\frac{k}{2}\right\rfloor$. 

\begin{claim}\label{claim:number_of_0_and_d_cells} 
	$|\mathrm{supp}_{0}(w)|\leq|\mathrm{supp}_{d}(w)|+d$ for every word $w$. 
\end{claim}

\begin{proof}
	The word $w$ starts in a $(d-1)$-cells which contains $d$ distinct $0$-cells. In order to obtain a new $0$-cells in the $i$-th crossing, one must observe in the crossing $\sigma_{i}\sigma_{i+1}$ a new $d$-cell. Since there are $|\mathrm{supp}_{d}(w)|$ distinct $d$-cells in $w$ the result follows. 
\end{proof}

Consequently, $d\leq|\mathrm{supp}_{0}(w)|\leq|\mathrm{supp}_{d}(w)|+d\leq\left\lfloor \frac{k}{2}\right\rfloor +d$.  Using (\ref{eq:semicircle_for_H_5}) together with the last claim we thus conclude that 
\begin{align}
	\mathbb{E}\left[\int_{\mathbb{R}}x^{k}L_{H}(\mathrm{d}x)\right] & =\sum_{s=d}^{\left\lfloor \frac{k}{2}\right\rfloor +d}\sum_{w\in\mathcal{W}_{s}^{k}}|[w]|\, T(w)\nonumber \\
 & =\sum_{s=d}^{\left\lfloor \frac{k}{2}\right\rfloor +d}\frac{1}{N(nq)^{k/2}}\sum_{w\in\mathcal{W}_{t,s}^{k}}|[w]|\prod_{\tau\in X^{d}}\mathbb{E}\left[(\chi-p)^{N_{w}(\tau)}\right]\mathrm{sgn}(w,\tau).\label{eq:semicircle_H_6}
\end{align}

We record the following simple claims, whose proofs are straightforward. 

\begin{claim}\label{claim:B_n_s_d}
	For every $w\in\mathcal{W}_{s}^{k}$, the number of elements in $[w]$ is $B_{n,s,d}:=\frac{n(n-1)\cdots(n-s)}{d!}$ and in particular is bounded by $n^{s}$.
\end{claim}

\begin{claim}\label{claim:bound_on_W_t_s}
	For every $k\geq0$ and $d\leq s\leq\left\lfloor \frac{k}{2}\right\rfloor +d$ we have
	\begin{equation}
		|\mathcal{W}_{s}^{k}|\leq k^{dk}.\label{eq:semicircle_H_9}
\end{equation}
\end{claim}

Also, observe that for $w\in\mathcal{W}_{s}^{k}$ 
\begin{align}\label{eq:semicircle_H_8}
	\prod_{\tau\in X^{d}}\mathbb{E}\left[(\chi-p)^{N_{w}(\tau)}\right]   & = \prod_{\substack{\tau\in X^{d}\\ N_{w}(\tau)\geq 2}} p\left(1-p\right)\left[(1-p)^{N_{w}(\tau)-1}+p^{N_{w}(\tau)-1}\right]\\ \nonumber
& \leq\prod_{\substack{\tau\in X^{d}\\ N_{w}(\tau)\geq2}} p(1-p)=q^{|\mathrm{supp}_{d}(w)|}\leq q^{|\mathrm{supp}_{0}(w)|-d}=q^{s-d},
\end{align}
where for the last inequality we used Claim \ref{claim:number_of_0_and_d_cells}.

Combining \eqref{eq:semicircle_H_8}, Claim \ref{claim:B_n_s_d} and Claim \ref{claim:bound_on_W_t_s} we conclude that for odd $k$ and large enough $n$
\begin{multline*}
	  \left|\sum_{s=d}^{\left\lfloor \frac{k}{2}\right\rfloor +d}\frac{B_{n,s,d}}{N(nq)^{k/2}}\sum_{w\in\mathcal{W}_{s}^{k}}\prod_{\tau\in X^{d}}\mathbb{E}\left[(\chi-p)^{N_{w}(\tau)}\right]\mathrm{sgn}(w,\tau)\right|
\leq  \sum_{s=d}^{\left\lfloor \frac{k}{2}\right\rfloor +d}\frac{n^{s}q^{s-d}}{N(nq)^{k/2}}|\mathcal{W}_{s}^{k}|
\\
\leq\sum_{s=d}^{\left\lfloor \frac{k}{2}\right\rfloor +d}\frac{(nq)^{s}}{Nq^{d}(nq)^{k/2}} k^{dk}
\overset{(1)}{\leq}  \sum_{s=d}^{\left\lfloor \frac{k}{2}\right\rfloor +d}\frac{(nq)^{\left\lfloor \frac{k}{2}\right\rfloor +d}}{Nq^{d}(nq)^{k/2}} k^{dk}\leq\frac{k}{2} k^{dk} d!(nq)^{\left\lfloor \frac{k}{2}\right\rfloor -\frac{k}{2}}=O_{k,d}\left(\frac{1}{(nq)^{1/2}}\right),
\end{multline*}
where for $\left(1\right)$ we used that $\lim_{n\to\infty}nq=\infty$, and in particular $nq\geq1$ for large enough $n$. This completes the proof for odd $k$.

Similarly, when $k$ is even, one can separate the sum to $s=\frac{k}{2}+d$
which we denote by $R_{1}$ and the sum over $s<\frac{k}{2}+d$ which
we denote by $R_{2}$, and obtain
\[
	R_{2}=O_{k,d}\left(\frac{1}{nq}\right).
\]
Thus, for $k$ even we have
\begin{align*}
	\mathbb{E}\left[\int_{\mathbb{R}}x^{k}L_{H}(\mathrm{d}x)\right] & =R_{1}+O_{k,d}\left(\frac{1}{nq}\right)\\
 & =\frac{B_{n,\frac{k}{2}+d,d}}{N(nq)^{k/2}}\sum_{w\in\mathcal{W}_{k/2+d}^{k}}\prod_{\tau\in X^{d}}\mathbb{E}\left[(\chi-p)^{N_{w}(\tau)}\right]\mathrm{sgn}(w,\tau)+O_{k,d}\left(\frac{1}{nq}\right).
\end{align*}

The following Lemma contains the key estimates needed to complete the proof.  

\begin{lem}\label{lem:Size_of_W_k/2_k/2+d} 
	The following claims hold for even $k$ and $w\in\mathcal{W}_{k/2+d}^{k}$.
	\begin{enumerate}
		\item $N_{w}(\tau)\in\{ 0,2\} $ for every $d$-cell $\tau$.
		\item $|E_{w}|=\frac{k}{2}$ and $N_{w}(e)=2$ for every $e\in E_{w}$. In particular $\mathrm{sgn}(w,\tau)=1$ for every $\tau\in X^{d}$. 
		\item $|\mathcal{W}_{k/2+d}^{k}|=\mathcal{C}_{\frac{k}{2}}d^{k/2}$.
\end{enumerate}
\end{lem}

Assuming Lemma \ref{lem:Size_of_W_k/2_k/2+d}, the proof of Lemma \ref{lem:Semicircle_moment_lemma} is now complete. Indeed, using the first two equalities in (\ref{eq:semicircle_H_8}) together with Lemma \ref{lem:Size_of_W_k/2_k/2+d} yields for even $k$
\begin{equation}
	\mathbb{E}\left[\int_{\mathbb{R}}x^{k}L_{H}(\mathrm{d}x)\right]
=\frac{n(n-1)\cdots(n-k/2-d)}{d! N(nq)^{k/2}}\,\mathcal{C}_{\frac{k}{2}}d^{k/2} q^{\frac{k}{2}}+O_{k,d}\left(\frac{1}{nq}\right),
\end{equation}
which gives 
\[
	\lim_{n\to\infty}\mathbb{E}\left[\int_{\mathbb{R}}x^{k}L_{H}(\mathrm{d}x)\right]=\mathcal{C}_{\frac{k}{2}}d^{k/2}.
\]

What remains, therefore, is the proof of Lemma \ref{lem:Size_of_W_k/2_k/2+d}, which  contains the main novelty of the proof. In the graph case, $d = 1$, the fact that each $d$-cell is crossed either
twice or not at all implies the same for the edges of the graph $G_{w}$. For $d \geq 2$ this is no longer true, and it is not immediate that part (1) of
Lemma \ref{lem:Size_of_W_k/2_k/2+d} implies part (2), as each $d$-cell
can be associated with up to $\binom{d+1}{2}$ different edges of
$G_{w}$.

\begin{proof}[Proof of Lemma \ref{lem:Size_of_W_k/2_k/2+d}]
	Fix $w=\sigma_{1}\sigma_{2}\ldots\sigma_{k}\sigma_{k+1}\in\mathcal{W}_{k/2+d}^{k}$
with $\sigma_{k+1}=\sigma_{1}$. We begin with the proof of (1). As observed before, the combination of the upper bound $|\mathrm{supp}_{d}(w)|\leq\frac{k}{2}$ with Claim \ref{claim:number_of_0_and_d_cells} implies that for $w\in\mathcal{W}_{k/2+d}^{k}$ we have $|\mathrm{supp}_{d}(w)|=\frac{k}{2}$, which means that there are exactly $\frac{k}{2}$ distinct $d$-cells such that $N_{w}(\tau)\geq 2$ (recall that it is impossible to have $N_{w}(\tau)=1$). Since $w$ is also a closed word of length $k+1$ we conclude that 
\[
	\frac{k}{2}\cdot 2\leq\sum_{\tau\in X^{d}}N_{w}(\tau)=\sum_{e\in E_{w}}N_{w}(e)=k
\]
and therefore $N_{w}(\tau)=2$ for each of the $d$-cells in $\mathrm{supp}_{d}(w)$. This concludes the proof of (1).

Next, we prove (2). Let $\tau_{1},\ldots,\tau_{\frac{k}{2}}$ be the $d$-cells crossed along the path generated by $w$ in the order of their appearance. Since the $d$-cells appear along the path generated by $w$, we must have that for every $1\leq i\leq\frac{k}{2}-1$ the $d$-cell $\tau_{i+1}$ is attached to one of the $d$-cells $\tau_{1},\ldots,\tau_{i}$ along a joint $(d-1)$-cell in their boundary, and in particular $\tau_{i+1}$ can add at most one new $0$-cell to the $0$-cells in $\tau_{1}\cup\ldots\cup\tau_{i}$ . Noting that $\mathrm{supp}_{0}(w)=\tau_{1}\cup\ldots\cup\tau_{\frac{k}{2}}$, and recalling that for $w\in\mathcal{W}_{k/2+d}^{k}$ we have $|\mathrm{supp}_{0}(w)|=\frac{k}{2}+d$, it follows that for every $1\leq i\leq\frac{k}{2}$ the $d$-cell $\tau_{i+1}$ contains exactly one $0$-cell that does not belong to $\tau_{1}\cup\ldots\cup\tau_{i}$.

Next, we show that $|V_{w}\cap\partial\tau_{i}|=2$ for every $1\leq i\leq\frac{k}{2}$, and thus in particular that the number of edges in $E_{w}$ is $\frac{k}{2}$, each of which is crossed precisely twice. The argument is separated into three cases, one of which is then further splitted into three subcases. Assume that there exists $\tau\in\mathrm{supp}_{d}(w)$ such that $|V_{w}\cap\partial\tau|>2$ and let $1\leq i\leq\frac{k}{2}$ be the minimal index such that $|V_{w}\cap\partial\tau_{i}|>2$.  Assume further that $\sigma_{l}\sigma_{l+1}$ is the first crossing of $\tau_{i}$ and $\sigma_{m}$ is the first appearance of a $(d-1)$-cell
in $\partial\tau_{i}$ which is not $\sigma_{l}$ or $\sigma_{l+1}$. Note that it is impossible to have $m=l$ or $m=l+1$. 

\begin{itemize}
	\item \textbf{Case 1:} $m<l$ (see Figure \ref{fig:Cases_of_tree_lemma}(a)). This case is impossible since it implies that the first crossing of $\tau_{i}$ does not add a new $0$-cell to the ones in $\tau_{1}\cup\ldots \cup \tau_{i-1}$.  
	\item \textbf{Case 2: $m>l+1$} and the first time the $d$-cell $\sigma_{m-1}\cup\sigma_{m}$ appears along the path generated by $w$ is in the crossing $\sigma_{m-1}\sigma_m$ (see Figure \ref{fig:Cases_of_tree_lemma}(b)). This is impossible, because it means that the $d$-cell $\tau_{j}=\sigma_{m-1}\cup\sigma_{m}$ does not add a new $0$-cell to the ones in $\tau_{1}\cup\ldots\cup\tau_{j-1}$ as it should. 
	\item \textbf{Case 3: }$m>l+1$ and the $d$-cell $\tau_{j}=\sigma_{m-1}\cup\sigma_{m}$
already appeared along the path generated by $w$ before the crossing $\sigma_{m-1}\sigma_m$.
	\begin{itemize}
		\item \textbf{Case 3.1:} $j<i$. This is impossible, because it implies that $\tau_{i}$ is not the first $d$-cell with the property that $|V_{w}\cap\partial\tau|>2$ in the list. 
		\item \textbf{Case 3.2: }$j=i$ (see Figure \ref{fig:Cases_of_tree_lemma}(c.2)). This is only possible if $m=l+2$ (otherwise $m$ is not the first appearance after $l+1$ of a $(d-1)$-cell from the boundary of $\tau_{i}$). However, in this case one can define a new word $w'=\sigma_{l+2}\ldots\sigma_{k+1}\sigma_{2}\ldots\sigma_{l}\sigma'\sigma_{l}$, where $\sigma'$ is a $(d-1)$-cell which is a neighbor of $\sigma_{l}$ and contains a $0$-cell that does not belong $\mathrm{supp}_{0}(w)$. The word $w'$ has the same number of $d$-cells as the original word $w$, namely $|\mathrm{supp}_{d}(w')|=\frac{k}{2}$, but has one more $0$-cell than $w$, i.e., $|\mathrm{supp}_{0}(w')|=\frac{k}{2}+d+1$. This however contradicts Claim \ref{claim:number_of_0_and_d_cells}.
		\item \textbf{Case 3.3: }$j>i$ (see Figure \ref{fig:Cases_of_tree_lemma}(c.3)). Denote by $\{ \sigma_{r},\sigma_{r+1}\} $ the first edge crossed in the $d$-cell $\tau_{j}$ with $l+1<r<m$. This gives yet another contradiction since the first appearance of the $d$-cell $\tau_{j}$ did not add a new $0$-cell to the ones in $\tau_{1}\cup\ldots\cup\tau_{j-1}$
as it should (we know that $\sigma_{r+1}$ is obtained from $\sigma_{r}$ by removing one vertex and necessarily adding a vertex which belongs to $\tau_{i}$ or otherwise $\sigma_{m}$ does not belong to $\partial\tau_{i}$). 
	\end{itemize}
\end{itemize}

Since all cases leads to contradiction we conclude that $|V_{w}\cap \partial \tau_{i}|=2$ for every $1\leq i\leq\frac{k}{2}$ and in particular, there is exactly one edge associated with each $d$-cell, namely $|E_{w}(\tau)|=1$ for every $\tau\in\mathrm{supp}_{d}(w)$. This concludes the proof of (2).

\begin{figure}[!h]
	\begin{centering}
	\includegraphics[scale=0.6]{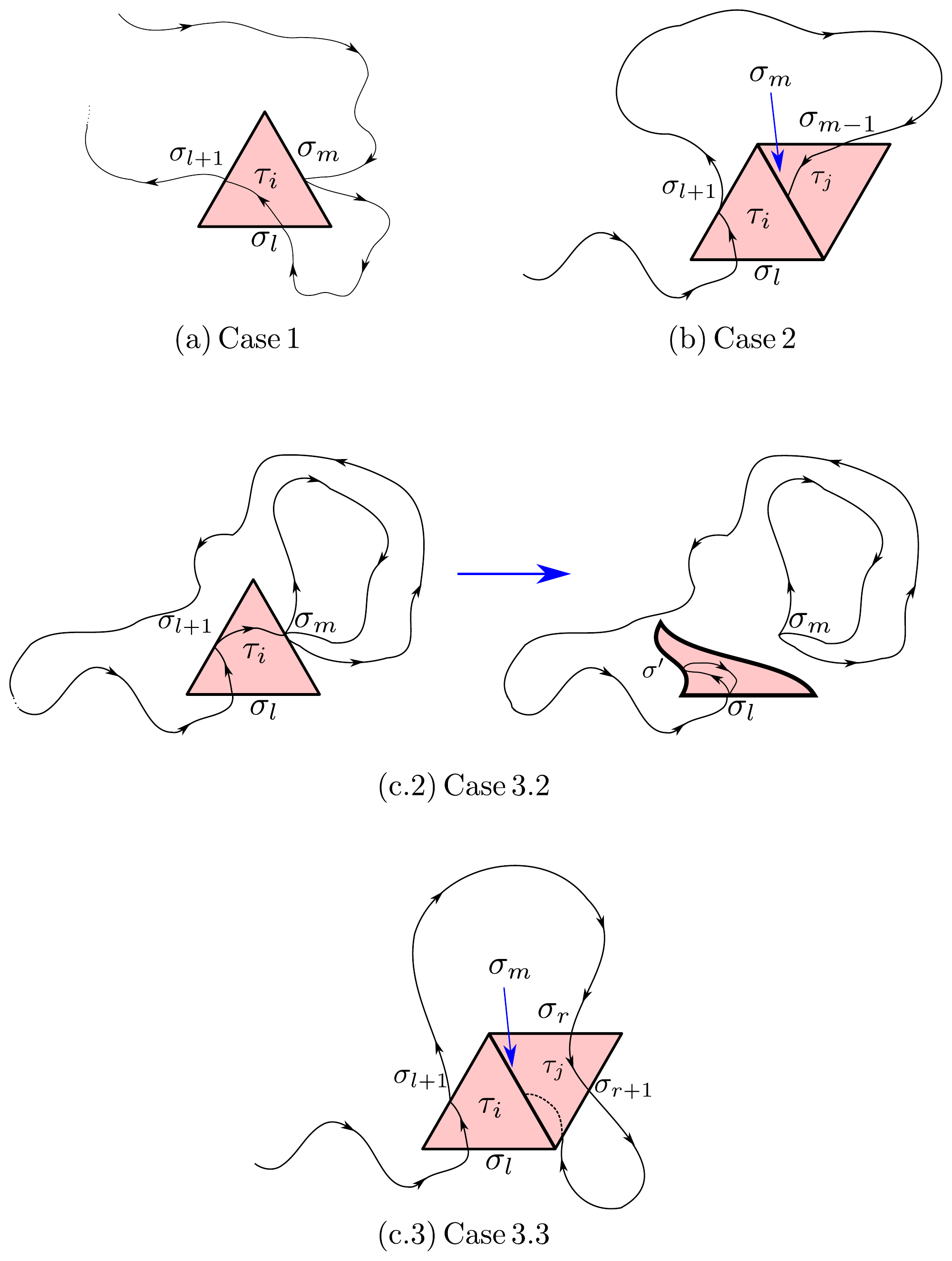}
	\caption{An illustration for the different cases of a $d$-cell $\tau$ with $|V_{w}\cap\partial\tau|>2$.
	\label{fig:Cases_of_tree_lemma}}
\par\end{centering}
\end{figure}

Finally, we prove (3). We start by showing that for $w\in\mathcal{W}_{k/2+d}^{k}$ the graph $G_{w}$ is a tree. Indeed, the graph $G_{w}$ is connected by definition and it cannot contain a loop since this would imply that one of the $d$-cells does not add a new $0$-cell in its first appearance. Since $G_{w}$ is a tree that comes with an additional path covering its vertices and edges we can think of $G_{w}$ as a rooted planar tree by declaring the first letter of $w$ to be its root and choosing a planar embedding of it that will make the path generated by $w$ into a clockwise exploration path of the tree (see \cite[Lemma 2.1.6]{AGZ10}).

Finally, we turn to evaluate $|\mathcal{W}_{k/2+d}^{k}|$. It is well known, see for example \cite[Lemma 2.1.6]{AGZ10}, that the set rooted planar trees with $\frac{k}{2}$ edges is in bijection with Dyck paths of length $k$ and is thus of size $\mathcal{C}_{k/2}$. We will show that the set $\mathcal{W}_{k/2+d}^{k}$ is in bijection with the set of rooted, labeled planar trees with $\frac{k}{2}$ edges or more formally the set of rooted planar trees with a label
from $\llbracket d\rrbracket $ attached to each of the tree edges. Since the total number of possible labelings is $d^{k/2}$ it follows that $|\mathcal{W}_{k/2+d}^{k}|=\mathcal{C}_{k/2}d^{k/2}$, thus completing the proof.

It is hence left to construct the aforementioned bijection. To this end, we label the $0$-cells of $X$ by the numbers in $\llbracket n\rrbracket $ and associate with every $j$-cell $\sigma=\{ \sigma^{0},\ldots,\sigma^{j}\} \in\binom{\llbracket n\rrbracket }{j+1}$ an ordering $\langle \sigma\rangle =(\sigma^{i_{0}},\ldots,\sigma^{i_{j}})$ such that $\sigma^{i_{0}}<\sigma^{i_{1}}<\ldots<\sigma^{i_{j}}$. In order to make the bijection simpler to write we choose to work with the specific choice of orientation $X_{+}^{j}$ associated with
the above ordering defined by $X_{+}^{j}=\{ [\sigma^{0},\ldots,\sigma^{j}]\,:\,\sigma^{0}<\ldots<\sigma^{j}\} $. Finally, we fix a special representative $w\in\mathcal{W}_{k/2+d}^{k}$ by requiring that the $0$-cells in $w$ are $\{ 1,\ldots,\frac{k}{2}+d\} $ and that they appear along the path $w$ in increasing order.

Given $w\in\mathcal{W}_{k/2+d}^{k}$ define its rooted, labeled planar tree $(G_{w},\ell_{w})$ by letting $G_{w}$ be the graph of the word $w$ as defined in Definition \ref{def:graph_of_a_word} with the planar embedding that makes $w$ into a clockwise exploration path. The labeling $\ell_{w}:E_{w}\to\llbracket d\rrbracket$ is then defined as follows: For an edge $\{ \sigma,\sigma'\} \in E_{w}$ let $1\leq i\leq k$ be the first integer such that $\{ \sigma_{i},\sigma_{i+1}\} =\{ \sigma,\sigma'\} $. Then we define $\ell(\{ \sigma,\sigma'\} )=j$ if and only if the $(d-2)$-cell $\sigma_{i}\cap\sigma_{i+1}$
is obtained from $\sigma_{i}$ by deleting the $j$-th smallest $0$-cell in it (see Figure \ref{fig:bijection_of_words_and_labeled_trees} for an illustration).

In the other direction, given a rooted, labeled planar tree
$(G=(V,E),\ell:E\to\llbracket d\rrbracket)$, define $w=\sigma_{1}\ldots\sigma_{k}\sigma_{1}\in\mathcal{W}_{k/2+d}^{k}$ by following procedure: 
\begin{itemize}
	\item Associate with the root of the tree the letter $\sigma_{1}=[1,\ldots,d]$. 
	\item Following the exploration path of the tree clockwise, if in the $i$-th step of the exploration the explored vertex appeared before, say in the $j$-th step, define $\sigma_{i}=\sigma_{j}$. If the $i$-th explored vertex is a new vertex then define $\sigma_{i}$ to be $\sigma_{i}=\sigma_{i-1}\backslash\sigma_{i-1}^{l(\{ \sigma_{i-1},\sigma_{i}\} )}\cup\{ m\} $,
where $m=|\sigma_{1}\cup\sigma_{2}\cup\ldots\cup\sigma_{i-1}|+1$. That is, $\sigma_{i}$ is obtained from $\sigma_{i-1}$ by deleting the $\ell(\{ \sigma_{i-1},\sigma_{i}\})$-th smallest number in $\sigma_{i-1}$ (with $\ell(\{ \sigma_{i-1},\sigma_{i}\})$ being the label of the currently explored edge) and adding a new $0$-cell with the smallest number that did not appear so far. 
\end{itemize}

\begin{figure}[h]
	\begin{centering}
	\includegraphics[scale=0.7]{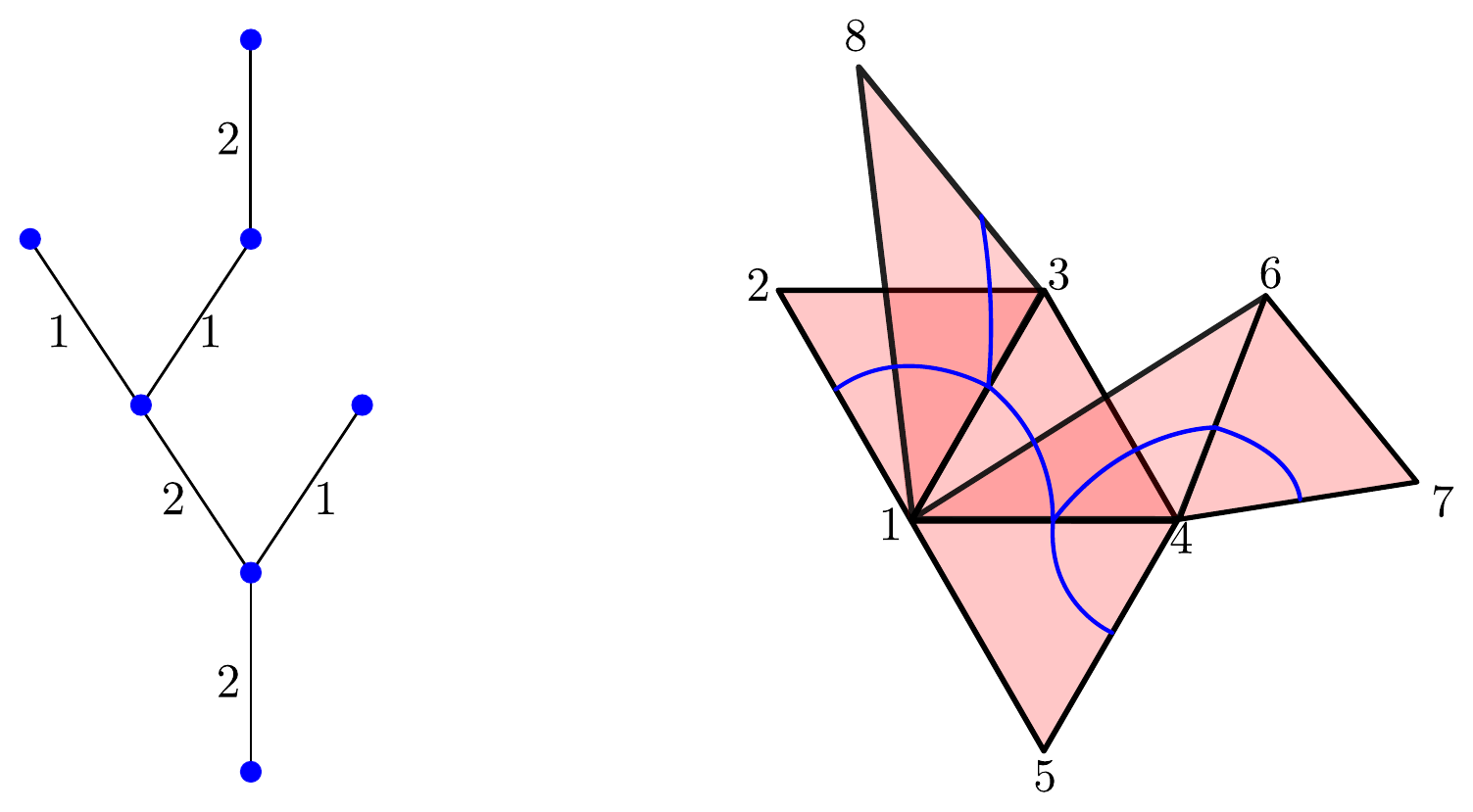}
	\par\end{centering}
	\centering{}\caption{(a) An illustration of the labeled tree associated with the word $w=[1,2][1,3][1,4][4,5][1,4][4,6][4,7][4,6][1,4][1,3][3,8][1,3][1,2]$. (b) The path of the word $w$ on the complex. 
	\label{fig:bijection_of_words_and_labeled_trees}}
\end{figure}

One can easily verify that these two maps are inverses of each other, thus completing the proof. 
\end{proof}


\section{Proof of Theorem \ref{thm:Semi_circle_law}} \label{sec:pf_thm26}

Since the eigenvalues of $(nq)^{-1/2}(A+pdI)$ and $\left(nq\right)^{-1/2}A$ are simplify shifted by $\frac{dp}{\sqrt{nq}}$, which by assumption tends to zero as $n \to \infty$, it is enough to prove the semicircle law for the matrix $(nq)^{-1/2}(A+pdI)$.

Using (\ref{eq:The_matrix_H}) and \eqref{eq:expectation_of_adj_matrix} we can write 
\begin{equation}
	\frac{1}{\sqrt{nq}}(A+pdI) 
= H+\frac{p}{\sqrt{nq}}(\mathbb{A}+dI).
	\label{eq:decomposition_of_A+dI}
\end{equation}

\begin{lem}[\cite{GW14} Lemma 8] \label{lem:eigenvalues_of_BA}
	The eigenvalues of $\mathbb{A}$ are $n-d$ with multiplicity $\binom{n-1}{d-1}$ and $-d$ with multiplicity $\binom{n-1}{d}$. 
\end{lem}

It follows that the matrix $\mathbb{A}+dI$ has rank $\binom{n-1}{d-1}$ and therefore by Weyl's interlacing inequalities we have 
\begin{equation}
	\lambda_{i}(H)\leq\lambda_{i}\left(\frac{1}{\sqrt{nq}}(A+pdI)\right)\leq\lambda_{i+\binom{n-1}{d-1}}(H),\quad\forall i\in\mathbb{Z},\label{eq:weyls_inequality_application}
\end{equation}
where for a self-adjoint, $N\times N$ matrix $B$ we define $\lambda_{i}(B)=-\infty$
for $i<1$ and $\lambda_{i}(B)=+\infty$ for $i>N$. 

For an $N\times N$ matrix $B$ denote by $\kappa_{B}(x) = L_B((-\infty,x])$ the cumulative distribution function of its eigenvalues. Let $f$ be a smooth bounded function with a bounded derivative. Then, using integration by parts we get
\begin{multline}\label{eq:integration_by_parts}
\int_{\mathbb{R}}f(E) \,  \mathrm{d} \kappa_{(nq)^{-1/2}(A+pdI)}(E)  -\int_{\mathbb{R}}f(E) \,  \mathrm{d} \kappa_{H}(E) 
\\
= \int f'(E)\left(\kappa_{H}(E)-\kappa_{(nq)^{-1/2}(A+pdI)}(E)\right)\mathrm{d}E.
\end{multline}
Due to (\ref{eq:weyls_inequality_application}) the right hand side is bounded by 
\[
\frac{\Vert f'\Vert _{\infty}}{N}\,\binom{n-1}{d-1}=O\left(\frac{\Vert f'\Vert _{\infty}}{n}\right),
\]
and thus \eqref{eq:integration_by_parts} goes to zero as $n\to\infty$.

It now follows by an approximation argument that
\begin{equation}
\lim_{n\to\infty}\left|\int_{\mathbb{R}}f(E) \, \mathrm{d}\kappa_{(nq)^{-1/2}(A+pdI)}(E)-\int_{\mathbb{R}}f(E) \, \mathrm{d}\kappa_{H}(E)\right|=0
\label{eq:proof_of_semicircle_for_shifted_A}
\end{equation}
for every bounded continuous function $f$. Therefore, by Theorem \ref{thm:Semicircle_for_H}
\[
\lim_{n\to\infty}L_{(nq)^{-1/2}(A+pdI)}=\lim_{n\to\infty}L_{H}=\nu_{d}
\]
almost surely (where the limits are in distribution), and Theorem \ref{thm:Semi_circle_law} follows.\hfill $\square$


\section{Bounding the norm of $H$ \label{sec:Bounding-the-norm-of-H}}

In this section we prove the following bound on the norm of the matrix $H$ defined in \eqref{eq:The_matrix_H}. 
\begin{thm}\label{thm:norm_bound}
Suppose that $nq \geq 2$. Then for every $\xi>0$ we have
	\[
		\mathbb{P}\bigl(\Vert H\Vert >2\sqrt{d}+\xi\bigr)\leq \mathcal{E}(\xi),
	\]
	where we introduced the error probability
	\begin{equation}\label{eq:norm_of_H_prob}
		\mathcal{E}(\xi) \equiv \mathcal{E}_{n,p,d}(\xi) := \frac{2\left(1+\frac{\xi}{2\sqrt{d}}\right)^{2}}{(d-1)!}\,\exp\left(d\log n-\left(\frac{2}{3}\log\left(1+\frac{\xi}{2\sqrt{d}}\right)\right)^{3/2}\left(\frac{nq}{d}\right)^{1/4}\right).
	\end{equation}
\end{thm}

\begin{rem} \label{rem:conv_eig}
	In particular, if $\lim_{n\to\infty}\frac{nq}{\log^{4}n}=\infty$ it follows that from the Borel-Cantelli lemma that for every $\varepsilon>0$ we have almost surely $\limsup_{n\to\infty}\Vert H\Vert <2\sqrt{d}+\varepsilon$. Since $\varepsilon$ was arbitrary we conclude that $\limsup_{n\to\infty}\Vert H\Vert \leq2\sqrt{d}$ almost surely. Observing that the semicircle law (Theorem \ref{thm:Semicircle_for_H}) implies that almost surely $\liminf_{n\to\infty}\Vert H\Vert \geq2\sqrt{d}$ we obtain that almost surely $\lim_{n\to\infty}\Vert H\Vert =2\sqrt{d}$.
\end{rem}

The proof of Theorem \ref{thm:norm_bound} uses a F\"uredi-Koml\'os-type argument; see \cite{FK81,AGZ10} for a presentation of the classical F\"uredi-Koml\'os argument for matrices with independent entries. The main work is to estimate the number of equivalence classes in $\mathcal{W}_{s}^{k}$. This estimate is given in the following result. 
\begin{prop}
\label{prop:Main_lemma_for_norm_bound_on_H} For every $k\geq0$ and $d\leq s\leq\lfloor \frac{k}{2}\rfloor +d$ 
\[
	|\mathcal{W}_{s}^{k}|\leq d (2\sqrt{d})^{k}\sum_{m=0}^{k-2(s-d)}\frac{\bigl(\frac{\sqrt{d}}{2}k^{3}\bigr)^{m}}{m!}.
\]

\end{prop}
Equipped with Proposition \ref{prop:Main_lemma_for_norm_bound_on_H},
whose proof is postponed, the proof of Theorem
\ref{thm:norm_bound} is standard.
\begin{proof}[Proof of Theorem \ref{thm:norm_bound}]
 Using (\ref{eq:semicircle_H_6}) we can write for even $k$  
\begin{align*}
	\mathbb{E}\left[\int_{\mathbb{R}}x^{k}L_{H}(\mathrm{d}x)\right] 
	& =\sum_{s=d}^{\frac{k}{2}+d}\frac{B_{n,s,d}}{N(nq)^{k/2}}\sum_{w\in\mathcal{W}_{s}^{k}}	\prod_{\tau\in X^{d}}\mathbb{E}\left[(\chi-p)^{N_{w}(\tau)}\right]\mathrm{sgn}(w,\tau)\\
	 & \overset{(1)}{\leq}\sum_{s=d}^{\frac{k}{2}+d}\frac{n^{s-d}|\mathcal{W}_{s}^{k}|}{(nq)^{k/2}}\sup_{w\in\mathcal{W}_{s}^{k}}\left|\prod_{\tau\in X^{d}}\mathbb{E}\left[(\chi-	p)^{N_{w}(\tau)}\right]\right|\\
	 & \overset{(2)}{\leq}\sum_{s=d}^{\frac{k}{2}+d}\frac{|\mathcal{W}_{s}^{k}|}{(nq)^{k/2+d-s}}\\
	 & \overset{(3)}{\leq}d(2\sqrt{d})^{k}\sum_{m=0}^{k}\frac{\bigl(\frac{\sqrt{d}}{2}k^{3}\bigr)^{m}}{(nq)^{k/2}m!}\frac{(nq)^{\frac{k-m}{2}+1}-1}{nq-1}\\
	 & \leq d(2\sqrt{d})^{k}\frac{nq}{nq-1}\sum_{m=0}^{k}\frac{1}{m!}\left(\sqrt{\frac{d}{4nq}}k^{3}\right)^{m}\\
	 & \overset{(4)}{\leq}2d(2\sqrt{d})^{k}\exp\left(\sqrt{\frac{d}{4nq}}k^{3}\right),
\end{align*}
where for $(1)$ we used the fact that $N^{-1}B_{n,s,d} \leq n^{s-d}$, for $(2)$ we used the fact that the supremum is bounded by $q^{s-d}$ (see (\ref{eq:semicircle_H_8})), for $(3)$
we used Proposition \ref{prop:Main_lemma_for_norm_bound_on_H}, and for $(4)$ we used
the assumption $nq\geq2$.

Therefore, it follows from the Markov's inequality that for every $\xi>0$ and $k$ even 
\begin{align*}
	& \mathbb{P}\left(\Vert H\Vert >2\sqrt{d}+\xi\right)\\
& \leq  \mathbb{P}\left(\sum_{i=1}^{N}(\lambda_{i}(H))^{k}>(2\sqrt{d}+\xi)^{k}\right)=\mathbb{P}\left(\int_{\mathbb{R}}x^{k}L_{H}(\mathrm{d}x)>\frac{(2\sqrt{d}+\xi)^{k}}{N}\right)\\
&\leq  \frac{N\,\mathbb{E}\left[\int_{\mathbb{R}}x^{k}L_{H}(\mathrm{d}x)\right]}{(2\sqrt{d}+\xi)^{k}}\leq\frac{2n^{d}}{(d-1)!}\left(\frac{2\sqrt{d}}{2\sqrt{d}+\xi}\right)^{k}\exp\left(\sqrt{\frac{d}{4nq}}k^{3}\right)\\
&\leq \frac{2}{(d-1)!}\exp\left(d\log n+\left[\sqrt{\frac{d}{4nq}}k^{2}-\log\left(1+\frac{\xi}{2\sqrt{d}}\right)\right]k\right).
\end{align*}
The claim then follows by choosing $k=k(n)$ to be the largest even integer that is smaller than $\left(\frac{2}{3\sqrt{d}}\log\left(1+\frac{\xi}{2\sqrt{d}}\right)\right)^{1/2}(nq)^{1/4}$.
\end{proof}

The rest of this section is devoted to the proof of Proposition \ref{prop:Main_lemma_for_norm_bound_on_H}. We start with the following definitions, which provide the correct generalization of the classical definition of FK words \cite{AGZ10} to the case $d>1$. 

\begin{defn}[Sentences]\label{def:(n,d)-sentence} 
	An $(n,d)$-sentence (or shortly a \emph{sentence}) is a finite sequence of words $\left(w_{1},w_{2},\ldots,w_{m}\right)$ at least one word long. The \emph{length of a sentence} $a$ is the sum of the lengths of its words. Note that unlike for words, in general the length is not the same as the number of crossings minus one. 
\end{defn}

\begin{defn}[Support of a sentence]
	For a sentence $a=(w_{1},\ldots,w_{m})$ we denote by $\mathrm{supp}_{0}(a)$ and $\mathrm{supp}_{d}(a)$ the union of $\mathrm{supp}_{0}(w_{i})$ respectively $\mathrm{supp}_{d}(w_{i})$ over $1\leq i\leq m$. 
\end{defn}

\begin{defn}[Graph of a sentence]\label{def:graph_of_a_sentence}
	Given a sentence $a=(w_{1},\ldots,w_{m})$ with $w_{i}=\sigma_{i,1}\sigma_{i,2}\ldots\sigma_{i,\ell_{i}}$ we define the graph associated with it $G_{a}=(V_{a},E_{a})$ by 
\[
	V_{a}=\{ \sigma_{i,j}\,:\,1\leq i\leq m,\,\,1\leq k\leq\ell_{i}\} 
\]
and 
\[
	E_{a}=\{ \{ \sigma_{i,j},\sigma_{i,j+1}\} \,:\,1\leq i\leq m,\,\,1\leq j\leq\ell_{i}-1\} .
\]

As in the word case the sentence $a$ induces a sequence of paths on the graph $G_{a}$ which together cover all of its vertices and edges. Also, the paths induces an ordering on the $(d-1)$-cells, $0$-cells and $d$-cells associated with the graph by following the words according to their order in the sentence and following the usual order inside each word. As in the word case we define $N_{a}(e)=\sum_{i=1}^{m}N_{w_{i}}(e)$ to be the number of times the edge $e\in E_{a}$ is crossed in the sentence $a$, let $E_{a}(\tau)=\{ \{ \sigma,\sigma'\} \in E_{a}\,:\,\sigma\cup\sigma'=\tau\}$ be the set of edges associated with the $d$-cell $\tau$ and define $N_{a}(\tau)=\sum_{e\in E_{a}(\tau)}N_{a}(e)$ to be the total number of crossings of $\tau$. 
\end{defn}

\begin{defn}[Wigner words]
	A closed word $w$ of length $k+1\geq1$ is called a \emph{Wigner word} if either $k=0$ or $k>0$ is even and $w$ is equivalent to an element of $\mathcal{W}_{k/2+d}^{k}$.
\end{defn}

\begin{defn}[FK words]\label{def:FK-words and FK-sentences} 
	A word $w$ is called an \emph{FK word} if the graph $G_{w}$ associated with $w$ is a tree, $N_{w}(e)\leq 2$ for every $e\in E_{w}$, and for each $\tau\in X^{d}$ there is a most one edge $e\in E_{w}(\tau)$ such that $N_{w}(e)=2$. More generally, a sentence $a=(w_{1},\ldots,w_{m})$ is called an \emph{FK sentence} if the graph $G_{a}$ is a tree, $N_{a}(e)\leq 2$ for every $e\in E_{a}$, for each $\tau\in X^{d}$ there is at most one edge $e\in E_{a}(\tau)$ such that $N_{a}(e)=2$, and for every $1\leq i\leq m-1$ the first letter of $w_{i+1}$ belongs to one of the words $w_{1},\ldots,w_{i}$. Two FK sentences $a=(w_{1},\ldots,w_{m})$ and $a'=(w_{1}',\ldots,w_{m'}')$ are called \emph{equivalent} if $m = m'$ and there exists a permutation $\pi$ on $X^0 = V$ such that $\pi(w_i) = \pi(w_i')$ for $1 \leq i \leq m$ (see Definition \ref{def:letter_words_and_equivalence} for the meaning of $\pi(w_i)$).
\end{defn}

\begin{defn}[Word parsing]
	A \emph{parsing} of a word $w=\sigma_{1}\ldots\sigma_{k}$ is a sentence $a=(w_{1},\ldots,w_{m})$ such that the concatenation of its words yields $w$. We say that $w=\sigma_{1}\ldots\sigma_{k}$ is parsed at time $i$ (or parsed in $\sigma_{i}\sigma_{i+1}$) in the parsing $a$ if $\sigma_{i}$ and $\sigma_{i+1}$ do not belong to the same word in $a$. 
\end{defn}

The key definition above is that of the FK word, which is, as it turns out, the correct generalization of the definition in \cite[Section 2.1.6]{AGZ10} to $d>1$. 

Before starting the formal proof of the lemma let us explain its general scheme; see also \cite{AGZ10} for more explanations on the structure of the F\"uredi-Koml\'os argument for matrices with independent entries. The proof starts by showing that each word $w\in\mathcal{W}_{s}^{k}$ can be parsed into an FK sentence. Since the original word can be read from its parsing simply by concatenating the words, it follows that it is enough to bound the number of such FK sentences. This is obtained via the following three steps. (a) Bounding the number of words in the FK sentence. (b) Bounding the number of ways to choose the equivalence classes of the FK words for the FK sentence. (c) Bounding the number of ways one can ``glue'' the chosen FK words together in order to obtain an FK sentence. The first bound, (a), which is obtained for $d=1$ by a simple observation (see \cite{AGZ10}), requires a new argument for $d > 1$, which is similar in spirit to the one used to prove Lemma \ref{lem:Size_of_W_k/2_k/2+d}(2).
The second bound, (b), is obtained by showing that each FK word can be parsed to a sentence which is comprised of disjoint Wigner words and then applying Lemma \ref{lem:Size_of_W_k/2_k/2+d}(3) to bound their number. Finally, the bound (c) on the number of ways to ``glue'' the FK words together is obtained by showing that a ``good gluing'' implies the existence of joint geodesic (see the proof of Lemma \ref{lem:How_to_concatenate_a_new_FK_sequence}).

\begin{proof}[Proof of Proposition \ref{prop:Main_lemma_for_norm_bound_on_H} assuming
three lemmas]
  Given a closed word $w$ of length $k+1$ we define its FK parsing $a_{w}$ as follows: Declare an edge $e$ of the associated graph $G_{w}$ \emph{new }if for some index $1\leq i\leq k$ we have $e=\{ \sigma_{i},\sigma_{i+1}\} $ and $\sigma_{i+1}\notin\{ \sigma_{1},\ldots,\sigma_{i}\} $. If an edge $e$ is not \emph{new}, then it is \emph{old }(note that old edges are the ones that create a loop in $G_{w}$ in their first crossing)\emph{. }Define $a_{w}$ to be the parsing obtained from $w$ by parsing in $\sigma_{i}\sigma_{i+1}$ if one of the following occurs:
\begin{itemize}
	\item $\{ \sigma_{i},\sigma_{i+1}\} $ is an old edge of $G_{w}$. 
	\item $\{ \sigma_{i},\sigma_{i+1}\} $ is a third or subsequent crossing of the edge $\{ \sigma_{i},\sigma_{i+1}\} $. 
	\item $\{ \sigma_{i},\sigma_{i+1}\} $ is a second crossing of the edge $\{ \sigma_{i},\sigma_{i+1}\} $ and there exists an edge $\{ \sigma,\sigma'\} \in E_{w}(\sigma_i\cup\sigma_{i+1})$ which is crossed twice in $\sigma_{1}\sigma_{2}\ldots\sigma_{i}$.
\end{itemize}

Since this parsing eliminates all loops in $G_{w}$, third and subsequent visits to edges and leaves in each $d$-cell at most one edge which is crossed twice the resulting sentence $a_{w}$ is indeed an FK sentence.

As indicated before, every word $w$ can be recovered from its FK sentence $a_{w}$. Therefore it is enough to bound the number of such FK sentences. The following three lemmas contain the required bounds for this purpose. 

\begin{lem}\label{lem:num_of_words_in_FK_parsing}
	For every $w\in\mathcal{W}_{s}^{k}$ the number of words $m_w$ in the FK sentence $a_{w}$ satisfies
	\[
		1\leq m_{w}\leq k+1-2(s-d).
	\]
\end{lem}

\begin{lem}\label{lem:FK_lemma_1} 
	There are at most $\frac{\sqrt{d}}{2}(2\sqrt{d})^{k}$ equivalence classes of FK words of length $k$. 
\end{lem}

\begin{lem}\label{lem:How_to_concatenate_a_new_FK_sequence} 
	Let $b$ be an FK sentence comprised of $m-1$ words of total length $l$ and $z$ an FK word. Then there are at most $l^{2}$ FK words $w$ equivalent to $z$ such that $(b,w)$ is an FK sentence. (Informally, there are at most $l^2$ ways to glue $z$ to $b$ so that the result is an FK sentence.)
\end{lem}

Assuming the last three lemmas, we turn to complete the proof. Assume first that $m_{w}=m$. There are $\binom{k}{m-1}$ ways to choose $m$-tuples $(l_{1},\ldots,l_{m})$ of positive  integers summing to $k+1$ which are the length of the words in the sentence $a_{w}$. Given the lengths of the words, by Lemma \ref{lem:FK_lemma_1} there are at most $\prod_{i=1}^{m}\bigl(\frac{\sqrt{d}}{2}(2\sqrt{d})^{l_{i}}\bigr)=\bigl(\frac{\sqrt{d}}{2}\bigr)^{m}(2\sqrt{d})^{k+1}$ ways to choose equivalence classes for each of these FK words. Finally,
due to Lemma \ref{lem:How_to_concatenate_a_new_FK_sequence}, there are at most $l_{1}^{2}\left(l_{1}+l_{2}\right)^{2}\cdots\left(l_{1}+\ldots+l_{m-1}\right)^{2}\leq k^{2\left(m-1\right)}$ ways to choose representatives in order to obtain an FK sentence. 

Combining all of the above, we find that there are at most 
\[
	\binom{k}{m-1}\left(\frac{\sqrt{d}}{2}\right)^{m}(2\sqrt{d})^{k+1}k^{2(m-1)}
\]
equivalence classes of FK sentences with $m$ words. 

Using now Lemma \ref{lem:num_of_words_in_FK_parsing} we conclude that the number of equivalence classes of FK sentences of length $k+1$ is bounded by 
\begin{multline*}
	 \quad \sum_{m=1}^{k+1-2(s-d)}\binom{k}{m-1}\left(\frac{\sqrt{d}}{2}\right)^{m}(2\sqrt{d})^{k+1}k^{2(m-1)}\\
	\leq \sum_{m=1}^{k+1-2(s-d)}\left(\frac{\sqrt{d}}{2}\right)^{m}(2\sqrt{d})^{k+1}\frac{k^{3(m-1)}}{(m-1)!}
	=  d(2\sqrt{d})^{k}\sum_{m=0}^{k-2(s-d)}\frac{\bigl(\frac{\sqrt{d}}{2}k^{3}\bigr)^{m}}{m!}, \qquad \qquad 
\end{multline*}
thus completing the proof. 
\end{proof}
Next we turn the proof of Lemmas \ref{lem:num_of_words_in_FK_parsing}--\ref{lem:How_to_concatenate_a_new_FK_sequence}.

\begin{proof}[Proof of Lemma \ref{lem:num_of_words_in_FK_parsing}]
	 Fix $w\in\mathcal{W}_{s}^{k}$ and let $a_{w}$ be the parsing of $w$ into an FK sentence consisting of $m_{w}$ words. Denoting by $r_{w}$ the number of crossings along $a_{w}$, we have $k+1=m_{w}+r_{w}$. Thus, it is enough to show that $r_{w}\geq2(s-d)$. Since $|\mathrm{supp}_{0}(w)|=s$ it follows that there are $(s-d)$ crossings in $w$, denoted by $(\sigma_{i(j)}\sigma_{i(j)+1})_{j=1}^{s-d}$, such that $\sigma_{i(j)+1}$ contains a $0$-cell that does not belong to $\sigma_{1}\cup\ldots\cup\sigma_{i(j)}$. Without loss of generality, we assume that the $0$-cells are indexed by $\llbracket n\rrbracket$ and that they are discovered in increasing order. Therefore in the crossing $\sigma_{i(j)}\sigma_{i(j)+1}$ we discover the $0$-cell $d+i(j)$. Since $\sigma_{i(j)+1}$
contains a new $0$-cell that did not appear in $\sigma_{1}\cup\ldots\cup\sigma_{i(j)}$
we must have that $\sigma_{i(j)}\sigma_{i(j)+1}$ is a first crossing of the edge $\{ \sigma_{i(j)},\sigma_{i(j+1)}\}$ and that it does not create a loop (i.e., it is a new edge). Consequently, $\sigma_{i(j)}\sigma_{i(j)+1}$ for $1\leq j\leq s-d$ are all crossings in $a_{w}$ as well. In addition, the crossings $\sigma_{i(j)}\sigma_{i(j)+1}$ for $1\leq j\leq s-d$
are all disjoint since the $d$-cells associated with them $\sigma_{i(j)}\cup\sigma_{i(j)+1}$
are disjoint. Thus, $r_{w}\geq s-d$. 

Next, for $1\leq j\leq s-d$ denote by $\sigma_{k(j)}\sigma_{k(j)+1}$ the second crossing in $w$ of the $d$-cell $\sigma_{i(j)}\cup\sigma_{i(j)+1}$ in $w$. Each of the crossings $\sigma_{k(j)}\sigma_{k(j)+1}$ which is also a crossing in $a_{w}$ increases $r_{w}$ by one and
so we only need to find an additional crossing of $a_{w}$ to replace
those crossings $\sigma_{k(j)}\sigma_{k(j)+1}$ in $w$ which are not crossings in $a_{w}.$ Fix $1\leq j\leq s-d$ such that $\sigma_{k(j)}\sigma_{k(j)+1}$ is not a crossing in $a_{w}$. Since $\sigma_{k(j)}\sigma_{k(j)+1}$ is the second crossing of the $d$-cell $\sigma_{i(j)}\cup\sigma_{i(j)+1}$ this can only happen if the edge $\{ \sigma_{k(j)},\sigma_{k(j)+1}\}$ is an old edge which in particular implies that $\{ \sigma_{k(j)},\sigma_{k(j)+1}\} \neq\{ \sigma_{i(j)},\sigma_{i(j)+1}\}$, i.e., there are at least $3$ distinct $(d-1)$-cells in
the boundary of the $d$-cell $\sigma_{i(j)}\cup\sigma_{i(j)+1}$ appearing along $w$. Let $\widetilde{\sigma}_{j}$ denote the first $(d-1)$-cell in the boundary of $\sigma_{i(j)}\cup\sigma_{i(j)+1}$ that appears in $w$ and is not $\sigma_{i(j)}$ or $\sigma_{i(j)+1}$. Finally, let $\sigma_{l(j)}\sigma_{l(j)+1}$ denote the first crossing in $w$ such that $\sigma_{l(j)+1}=\widetilde{\sigma}_{j}$. Since $\sigma_{l(j)+1}$ is the first appearance of $\widetilde{\sigma}_{j}$, the crossing $\sigma_{l(j)}\sigma_{l(j)+1}$ is not of an old edge and is a first visit to the edge $\{ \sigma_{l(j)},\sigma_{l(j)+1}\} $. Consequently, it is also a crossing in $a_{w}$. 

For $1\leq j\leq s-d$ define 
\[
	r(j)=
	\begin{cases}
		k(j) & \quad\text{if } \sigma_{k(j)}\sigma_{k(j)+1}\,\mbox{is a crossing in }a_{w}\\
		l(j) & \quad\text{if }\sigma_{k(j)}\sigma_{k(j)+1}\,\mbox{is not a crossing in }a_{w}
	\end{cases}.
\]

In order to show that $r_{w}\geq2(s-d)$ and thus to complete
the proof it is thus left to show that the crossing times $\{ i(j)\} _{j=1}^{s-d}\cup\{ r(j)\} _{j=1}^{s-d}$ are distinct. This is indeed the case, as can be seen by the following
observations:
\begin{itemize}
	\item It was already observed before that the times $\{ i(j)\} _{j=1}^{s-d}$ are all distinct as the associated crossings $\sigma_{i(j)}\sigma_{i(j)+1}$ form distinct $d$-cells. Similarly, for every $1\leq j_{1},j_{2}\leq s-d$ the crossing times $i(j_{1})$ and $r(j_{2})$
are distinct. Indeed, $i(j_{1})$ is the first crossing time in which the $0$-cell $j_{1}+d$ is observed. Since $r(j_{2})$ cannot be a crossing time in which a new $0$-cell is observed, it
follows that the times $i(j_1)$ and $r(j_2)$ must be distinct. 
	\item Finally, we claim that the times $\{r(j)\}_{j=1}^{s-d}$ are distinct. Indeed, assume that for some $1\leq j_{1}<j_{2}\leq s-d$ we have $r(j_{1})=r(j_{2})$.

\begin{itemize}
	\item If $r(j_{1})=k(j_{1})$ and $r(j_{2})=k(j_{2})$, then $r(j_{1})\neq r(j_{2})$ since the $d$-cells $\sigma_{k(j_{1})}\cup\sigma_{k(j_{1})+1}=\sigma_{i(j_{1})}\cup\sigma_{i(j_{1})+1}$ and $\sigma_{k(j_{2})}\cup\sigma_{k(j_{2})+1}=\sigma_{i(j_{2})}\cup\sigma_{i(j_{2})+1}$
associated with the crossings are distinct and thus in particular so are the crossing times. 
	\item If $r(j_{1})=k(j_{1})$ and $r(j_{2})=l(j_{2})$, then $r(j_{1})\neq r(j_{2})$ since the $d$-cell $\sigma_{l(j_{2})}\cup\sigma_{l(j_{2})+1}$ contains the $0$-cell $d+j_{2}$ which by definition does not belong to the $d$-cell $\sigma_{k(j_{1})}\cup\sigma_{k(j_{1})+1}=\sigma_{i(j_{1})}\cup\sigma_{i(j_{1})+1}$.
	\item If $r(j_{1})=l(j_{1})$ and $r(j_{2})=k(j_{2})$, then $r(j_{1})\neq r(j_{2})$ (see Figure \ref{fig:FK_case}). Indeed, assume that $r(j_{1})=r(j_{2})$. Since the $d$-cell $\sigma_{l(j_{1})}\cup\sigma_{l(j_{1})+1}$ is the $d$-cell in which the $0$-cell $d+j_{2}$ appears for the
first time we must have that $d+j_{2}$ belongs to $\sigma_{l(j_{1})}\cup\sigma_{l(j_{1})+1}$.
In addition, since $\sigma_{l(j_{1})+1}$ belongs to $\sigma_{i(j_{1})}\cup\sigma_{i(j_{1})+1}$
and $j_{1}<j_{2}$, it follows that the $0$-cell $d+j_{2}$ does not belong to $\sigma_{l(j_{1})+1}$ and so it must be in $\sigma_{l(j_{1})}$. However, this implies that $\sigma_{k(j_{2})}\cup\sigma_{k(j_{2})+1}=\sigma_{i(j_{2})}\cup\sigma_{i(j_{2})+1}$ is not the first $d$-cell containing the $0$-cell $d+j_{2}$, which contradicts its definition.

	\begin{figure}[h]
		\begin{centering}
		\includegraphics[scale=0.8]{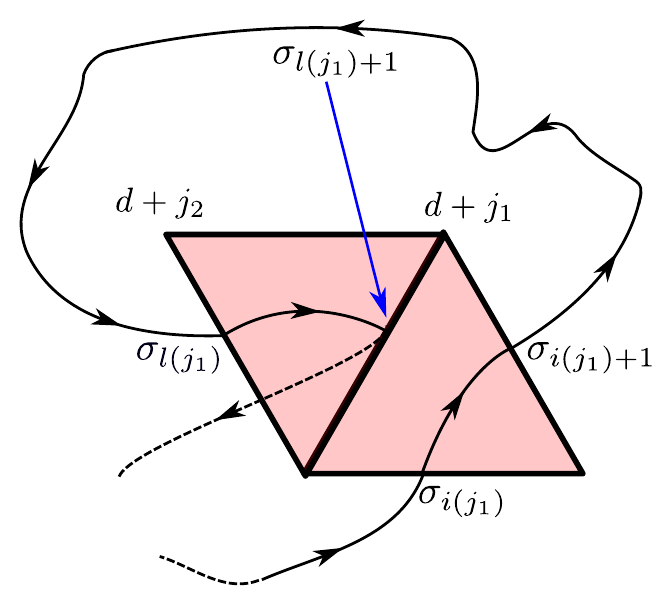}
		\caption{Illustration for the case $r(j_{1})=l(j_{1})$ and $r(j_{2})=k(j_{2})$. Note that the in order for $r(j_{1})$ to be the same as $r(j_{2})$ the path must reach the $(d-1)$-cell $\sigma_{l(j_{2})+1}$ (which does not contain the $0$-cell $d+j_{2}$) through a $d$-cell
that contains the $0$-cell $d+j_{2}$. However, this implies that this $d$-cell is not the first $d$-cell in which the $0$-cell $d+j_{2}$
appears. \label{fig:FK_case}}
		\par\end{centering}
	\end{figure}
 
	\item If $r(j_{1})=l(j_{1})$ and $r(j_{2})=l(j_{2})$, then $r(j_{1})\neq r(j_{2})$. Indeed, we must have that $\sigma_{l(j_{1})+1}$, which belongs to the boundary of $\sigma_{i(j_{1})}\cup\sigma_{i(j_{1})+1}$, only contains $0$-cells that appeared until the crossing of $\sigma_{i(j_{1})}\cup\sigma_{i(j_{1})+1}$, i.e., only $0$-cells from $1,2,\ldots,j_{1}+d$. However, $\sigma_{l(j_{2})+1}$ must contain the $0$-cell $d+j_{2}$ by definition. This leads to a contradiction since $j_{1}<j_{2}$. 
	\end{itemize}
\end{itemize}
This concludes the proof.
\end{proof}

\begin{proof}[Proof of Lemma \ref{lem:FK_lemma_1}]
	 We start by showing that each FK word can be written in a unique way as a concatenation of disjoint Wigner words. Let $w=\sigma_{1}\ldots\sigma_{l}$ be an FK word of length $l$. Let $\{\sigma_{i(j)},\sigma_{i(j)+1}\} _{j=1}^{r}$ be the edges in the graph $G_{w}$ which are crossed precisely once by the path $w$. Defining $i(0)=0$ and $w_{j}=\sigma_{i(j-1)+1}\ldots\sigma_{i(j)}$, we now claim that the words $w_{j}$ are closed and disjoint and that $G_{w_{j}}$ is a tree in which every $d$-cell, and every edge are crossed precisely twice, i.e., the words $w_{j}$ are disjoint Wigner words. Indeed, denote the above parsing of $w$ by $b_{w}$. Since $G_{w}$ is a tree, the graph $G_{b_{w}}$ is a forest and the paths generated by $b_{w}$ on it cross each edge precisely twice. In addition, since in an FK word there is at most one edge inside each $d$-cell which is crossed twice, it follows that each $d$-cell $\tau$ which
is crossed in $b_{w}$ is crossed precisely twice; both times on the unique edge $e\in E_{w}(\tau)$ such that $N_{w}(e)=2$. Finally, note that the words $w_{j}$ are closed since each edge in their graph is crossed precisely twice and it contains no loops as a subgraph of $G_{b_{w}}$. Thus each word must be associated with a unique connected component of $G_{b_{w}}$ and its path must cross each of the edges and $d$-cells of this component twice.  

From this point on the proof proceeds in the same way as in the case $d = 1$ (see \cite[Lemma 2.1.24]{AGZ10}), and we include the conclusion of the proof, adapted from \cite[Lemma 2.1.24]{AGZ10}, for the reader's convenience. The aforementioned decomposition
is unique and one concludes that with $N_{k}$ denoting the number of equivalence classes of FK words of length $k$ we get 
\begin{align*}
	\sum_{k=1}^{\infty}N_{k}z^{k} & =\sum_{r=1}^{\infty}\sum_{l_{1},l_{2},\ldots,l_{r}\in2\mathbb{N}}\prod_{j=1}^{r}z^{l_{j}+1}|\mathcal{W}_{l_{j}/2+d}^{l_{j}}|=\sum_{r=1}^{\infty}\left(z+\sum_{l=1}^{\infty}z^{2l+1}\left|\mathcal{W}_{l+d}^{2l}\right|\right)^{r}\\
 & =\sum_{r=1}^{\infty}\left(z+\sum_{l=1}^{\infty}z^{2l+1}\mathcal{C}_{l}d^{l}\right)^{r},
\end{align*}
where for the last step we used Lemma \ref{lem:Size_of_W_k/2_k/2+d}. Consequently, using the generating function of the Catalan numbers, for $|z|<\frac{1}{2\sqrt{d}}$ we have
\[
	\sum_{k=1}^{\infty}N_{k}z^{k}=\sum_{r=1}^{\infty}\left(\frac{1-\sqrt{1-4dz^{2}}}{2dz}\right)^{r}=\frac{1}{2}\left(-1+\frac{1+2dz}{\sqrt{1-4dz^{2}}}\right).
\]
Since $\frac{1}{\sqrt{1-t}}=\sum_{k=0}^{\infty}\frac{1}{4^{k}}\binom{2k}{k}t^{k}$
for $\left|t\right|<1$ it follows that 
\begin{align*}
	\sum_{k=1}^{\infty}N_{k}z^{k} & =-\frac{1}{2}+\frac{1}{2}(1+2dz)\sum_{k=1}^{\infty}d^{k}\binom{2k}{k}z^{2k}
\end{align*}
and thus 
\[
|N_{k}|\leq\frac{\sqrt{d}}{2}(2\sqrt{d})^{k}. \qedhere
\]
\end{proof}

With our definition of FK sentences, Lemma \ref{lem:How_to_concatenate_a_new_FK_sequence} follows by the same argument as in the case $d=1$ (see \cite[Lemma 2.1.25]{AGZ10}). For the reader's convenience we include the appropriately adapted version of this argument here. We start with some additional definitions. Recall from the proof of Lemma \ref{lem:FK_lemma_1} that any FK word $w$ can be written in a unique way as a concatenation of disjoint Wigner words $b_{w}=(w_{1},\ldots,w_{r})$. With $\sigma_{i,1}$ denoting the first (and last) letter of $w_{i}$, define the \emph{skeleton} of the FK word $w$ to be the word $\widetilde{w}=\sigma_{1,1}\sigma_{2,1}\ldots\sigma_{r,1}$. Finally, for an FK sentence $a$ with a graph $G_{a}$ define a new graph $\widetilde{G}_{a}=(\widetilde{V}_{a},\widetilde{E}_{a})$ by setting $\widetilde{V}_{a}=V_{a}$ and $\widetilde{E}_{a}=\{ e\in E_{a}\,:\,N_{a}(e)=1\}$. Since $G_{a}$ is a tree and the edges removed from it are the one corresponding to edges of the Wigner words $G_{w_{j}}$, the resulting graph $\widetilde{G}_{a}$ is a forest.

\begin{proof}[Proof of Lemma \ref{lem:How_to_concatenate_a_new_FK_sequence}]
As avertised above, the proof is almost identical to that of \cite[Lemma 2.1.25]{AGZ10}. We begin with the following observation. Suppose $b$ is an FK sentence with $m-1$ words and that $w$ is an FK word equivalent to $z$ such that $(b,w)$ is an FK sentence. Denote the skeleton of $w$ by $\widetilde{w}=\sigma_{j_{1}}\sigma_{j_{2}}\ldots\sigma_{j_{r}}$, so that $\sigma_{j_{1}}$ is a letter in $b$ by definition. Let $l$ be the largest index such that $\sigma_{j_{l}}$ is a letter in $b$ and set $w'=\sigma_{j_{1}}\ldots\sigma_{j_{l}}$. Then  $V_{b}\cap V_{w}=V_{w'}$ and $w'$ is a geodesic in $\widetilde{G}_{b}$. 

Assuming the observation holds we finish the proof. The number of ways to choose an FK word $w$ equivalent to $z$ such that $(b,w)$ is an FK sentence is bounded by the number of ways to choose a geodesic in its forest $\widetilde{G}_{b}$. The number of ways to choose a geodesic inside
a forest is bounded by the number of ways to choose its endpoints (since the geodesic is unique) which in turn is bounded by $l^{2}$. 

The proof of the observation follows from the same argument that is used in the graph case $d=1$; see \cite[Lemma 2.1.25]{AGZ10}. Suppose $a$ is an FK sentence. Then $G_{a}$ is a tree, and since the Wigner words composing $w$ are disjoint, $w'$ is the unique geodesic in
$G_{w}\subset G_{a}$ connecting $\sigma_{j_{1}}$ to $\sigma_{j_{l}}$. But $w'$ visits only edges of $G_{b}$ that have been visited exactly once by the words constituting $b$, for otherwise $(b,w)$ would not be an FK sentence (that is, a comme would need to be inserted
to split $w$). Thus $E_{w'}\subset\widetilde{E}_{b}$. Since $w$ is an FK word, $\widetilde{E}_{w}=E_{\widetilde{w}}$. Since $a$ is an FK sentence $E_{b}\cap E_{w}=\widetilde{E}_{b}\cap\widetilde{E}_{w}$. Thus $E_{b}\cap E_{w}=E_{w'}$. But, now recall that $G_{a},G_{b},G_{w},G_{w'}$ are all trees and hence 
\begin{align*}
	|V_{a}| & =1+|E_{a}|=1+|E_{b}|+|E_{w}|-|E_{b}\cap E_{w}|=1+|E_{b}|+|E_{w}|-|E_{w'}|\\
 & =1+|E_{b}|+1+|E_{w}|-1-|E_{w'}|=|V_{b}|+|V_{w}|-|V_{w'}|.
\end{align*}
Since $|V_{b}|+|V_{w}|-|V_{b}\cap V_{w}|=|V_{a}|$, it follows that $|V_{w'}|=|V_{b}\cap V_{w}|$. Since $V_{w'}\subset V_{b}\cap V_{w}$ one concludes that $V_{w'}=V_{b}\cap V_{w}$,
as claimed. 
\end{proof}


\section{Proof of Theorem \ref{thm:location_of_eigenvalues} } \label{sec:location_of_eigenvalues}

We start this section with the following stronger version of Theorem \ref{thm:location_of_eigenvalues}.
\begin{thm}
\label{thm:location_of_eigenvalues_strong_ver} Assume $d\geq2$ and
$nq\geq 10^3d^2$. 
\begin{enumerate}
\item For every $\xi>0$, the $\binom{n-1}{d}$ smallest eigenvalues of
the matrix $A$ are within the interval $-pd+\sqrt{nq}\,[-2\sqrt{d}-\xi,2\sqrt{d}+\xi]$ with probability at least $1-\mathcal{E}(\xi)$ (see \eqref{eq:norm_of_H_prob}).

\item For every $\xi>0$, if $nq\geq d(2d+2\xi)^6 n^{-1} \log^6n$,
then the remaining $\binom{n-1}{d-1}$ eigenvalues of $A$ are inside the interval $np+[-\Gamma(\xi,n),\Gamma(\xi,n)]$
with probability at least $1-\mathcal{E}((\sqrt{6}-2)\sqrt{d})-\mathscr{E}(\xi)$
where
\[
	\mathscr{E}(\xi)\equiv \mathscr{E}_{n,d}(\xi):=\frac{4e^{3}d^{5/2}}{(d-1)!}\exp(5\log(2d+2\xi)+5\log\log n-\xi\log n)
\]
and 
\begin{equation}
	\Gamma(\xi,n):=6d+\frac{200d^{3/2}}{\sqrt{nq}-24d} +100d^{7/2}(d+\xi)^{3}\sqrt{q}\log^{3}n.\label{eq:interval_size-1}
\end{equation}

\end{enumerate}
\end{thm}

Before proving Theorem \ref{thm:location_of_eigenvalues_strong_ver}, we show how it implies Theorem \ref{thm:location_of_eigenvalues}.

\begin{proof}[Proof of Theorem \ref{thm:location_of_eigenvalues} assuming Theorem
\ref{thm:location_of_eigenvalues_strong_ver}]
This is an elementary exercise, and we only sketch the argument. Part (1) of Theorem \ref{thm:main_result} is proved by dealing separately with the case $\xi<1$ and $\xi\geq 1$. In both cases one can show that the assumption $nq\geq\frac{C(1+D)^{4}}{\min\{ \xi^6,1\}}\log^{4}n$ implies that $\mathcal{E}(\xi/2)< n^{-D}$. Finally, note that under the same assumption we have $pd\leq \xi/(2 \sqrt{nq})$ which allows us to use the interval $\sqrt{dnq}[-2-\xi,2+\xi]$ instead of $-pd+\sqrt{dnq}[-2-\xi/2,2+\xi/2]$.
As for part (2) of Theorem \ref{thm:main_result}, note that the assumption $nq\geq C(1+D)^{4}\log^{4}n$ implies that $\mathcal{E}((\sqrt{6}-2)\sqrt{d})<n^{-D}/2$ for an appropriate choice of $C$. By choosing $\xi=C'(D+1)$ with $C'$ depending only on $d$ one can verify that $\mathscr{E}(C'(D+1))$ is bounded by $n^{-D}/2$ as well. Therefore as long as $nq \geq \frac{C'^2(D+1)^2 \log^6 n}{n}$ we have that with probability at least $1-n^{-D}$ that the remaining eigenvalues are within the interval $np+[-\Gamma(C'(1+D),n),\Gamma(C'(1+D),n)]$ which under both assumption of part (2) is contained within the interval $np+[-7d,7d]$ for an appropriate choice of the constant $C$. Finally, note that the case $nq\leq \frac{C'^2(D+1)^2 \log^6 n}{n}$ is impossible, since when combined with the assumption $nq \geq C(D+1)^4 \log^4 n$ implies that $(D+1)^2 \leq \frac{C'^2}{C} \frac{\log^2 n}{n}$, which by increasing the value of the constant $C$ does not hold for any value of $n$.
\end{proof}

The rest of this section is devoted to the proof of Theorem \ref{thm:location_of_eigenvalues_strong_ver}. We introduce 
\begin{equation}
	P:=\frac{1}{n}(\mathbb{A}+dI),
\end{equation}
which by Lemma \ref{lem:eigenvalues_of_BA} is an orthogonal projection of rank $\binom{n-1}{d-1}$, and
\begin{equation}
	\kappa := \frac{\sqrt{nq}}{1-p}.
\end{equation}

\begin{proof}[Proof of Theorem \ref{thm:location_of_eigenvalues_strong_ver} part
(1)] \label{Proof:proof_of_Theorem_location_part_1} 
Recall from (\ref{eq:decomposition_of_A+dI}) that 
\begin{align}
	(nq)^{-1/2}(A+pdI) & =H+\kappa P,		
	\label{eq:better_estimation_on_trivial_eigenvalues_1}
\end{align}
Therefore, it follows from Theorem \ref{thm:norm_bound} that for every $\xi>0$
with probability at least $1-\mathcal{E}(\xi)$ 
\[
	\Vert (nq)^{-1/2}(A+pdI)-\kappa P\Vert =\Vert H\Vert \leq2\sqrt{d}+\xi.
\]

Since $P$ is an orthogonal projection whose rank is $\binom{n-1}{d-1}$ it follows from Weyl's inequality that the eigenvalues of $(nq)^{-1/2}(A+pdI)$ are shifted by at most $2\sqrt{d}+\xi$ from the ones of $\kappa P$, which in particular gives the required result for the $\binom{n-1}{d}$ smallest eigenvalues of $A$.
\end{proof}

The above proof also implies that the remaining $\binom{n-1}{d-1}$ eigenvalues of $A$ are within an interval of size $\sqrt{nq}$ around $\kappa$ with probability at least $1-\mathcal{E}(\xi)$. Since we are interested in showing a better concentration result for
those eigenvalues we need to obtain better estimation on the largest $\binom{n-1}{d-1}$ eigenvalues. More precisely, we have the following. Let $Y:=(nq)^{-1/2}(A+pdI)$ and $\overline{P}:=I-P$, and split 
\[
	Y=\underset{Y_{1}}{\underbrace{(PYP+\overline{P}Y\overline{P})}}+\underset{Y_{2}}{\underbrace{(PY\overline{P}+\overline{P}YP)}}.
\]
It follows that with probability at least $1-\mathcal{E}(\xi)$ we have $\|Y_2\|\leq \|H\|\leq 2\sqrt{d}+\xi$ and that (due to \eqref{eq:better_estimation_on_trivial_eigenvalues_1}) the distance between the eigenvalues of $PYP=PHP+\kappa P$ and $\overline{P}Y\overline{P}=\overline{P}H\overline{P}$ is at least $\kappa - 2\|H\| \geq \kappa -2(2\sqrt{d}+\xi)$. Consequently, by \cite[Proposition A.1]{KY14},
\begin{equation}
	|\lambda_{i}(Y_{1})-\lambda_{i}(Y)|\leq\frac{(2\sqrt{d}+\xi)^{2}}{\kappa-4(2\sqrt{d}+\xi)}=O\left(\frac{1-p}{\sqrt{nq}}\right)
	\label{eq:better_estimation_on_trivial_eigenvalues_2}
\end{equation}
for every $\binom{n-1}{d}<i\leq\binom{n}{d}$. Since $Y_1= PHP +\kappa P +\overline{P} H\overline{P}$ it follows that $\lambda_{i}(Y_{1})$ for $\binom{n-1}{d}<i\leq\binom{n}{d}$
are shifted from $\kappa$ by an amount which is at most the norm
of the matrix $PHP$. 

Combining all of the above we conclude that 
\begin{equation}\label{eq:error_on_eigenvalue_dist_from_kappa}
\begin{aligned}
	|\lambda_i((nq)^{-1/2} A) - \kappa| & \leq |\lambda_i((nq)^{-1/2} A) - \lambda_i(Y)| 
	+ |\lambda_i(Y) - \lambda_i(Y_1)| + |\lambda_i(Y_1)-\kappa| \\ 
	& \leq \frac{pd}{\sqrt{nq}} + \frac{(2\sqrt{d}+\xi)^2}{\kappa-4(2\sqrt{d}+\xi)} + \|PHP\|
\end{aligned}
\end{equation}
for every $\binom{n-1}{d}<i\leq\binom{n}{d}$. Thus it is enough to obtain a better bound on the norm of $PHP$, which is the content of the following proposition. 

\begin{prop}\label{prop:bound_on_eigenvalues_of_PHP} 
	For every $\eta>0$, if $nq \geq \frac{\eta^2}{16d^6e^2} \frac{\log^6 n}{n}$, then 
\begin{equation}
\begin{aligned}
	& \mathbb{P}\left(\|PHP\| >\eta\frac{\log^{3}n}{\sqrt{n}}\right) \\
	& \qquad \leq  \frac{4e^{3}d^{5/2}}{(d-1)!}\exp\left(\frac{5}{3}\log\left(\frac{\eta}{4d^{7/2}e}\right)+5\log\log n-\left[\frac{1}{2}\left(\frac{\eta}{4d^{7/2}e}\right)^{1/3}-d\right]\log n\right).
\end{aligned}
\end{equation}
\end{prop}

Before proving Proposition \ref{prop:bound_on_eigenvalues_of_PHP}, we use it to conclude the proof of Theorem \ref{thm:location_of_eigenvalues_strong_ver}.

\begin{proof}[Proof of Theorem \ref{thm:location_of_eigenvalues_strong_ver} part
(2)]
 	It follows from Proposition \ref{prop:bound_on_eigenvalues_of_PHP} (with $\eta=4d^{7/2}e\,(2d+2\xi)^{3}$) that with probability at least $1-\mathscr{E}(\xi)$  we have $\|PHP\| \leq4d^{7/2}e\,(2d+2\xi)^{3}n^{-1/2}\log^{3}n$ assuming that $nq\geq d(2d+2\xi)^6 n^{-1} \log^6n$. When combined with \eqref{eq:error_on_eigenvalue_dist_from_kappa} (with the value $\xi=(\sqrt{6}-2)\sqrt{d}$ in Theorem \ref{thm:location_of_eigenvalues_strong_ver} part (1)) this yields that with probability at least $1-\mathcal{E}((\sqrt{6}-2)\sqrt{d})-\mathscr{E}(\xi)$ we have $|\lambda_{i}(A)- np|\leq \Gamma(\xi,n)$ for every $\binom{n-1}{d}<i\leq\binom{n}{d}$, assuming that $nq\geq d(2d+2\xi)^6 n^{-1}\log^6n$.
\end{proof}


\subsection{Proof of Proposition \ref{prop:bound_on_eigenvalues_of_PHP}}

The proof of Proposition \ref{prop:bound_on_eigenvalues_of_PHP} is based on a variant of the FK method developed in Section \ref{sec:Bounding-the-norm-of-H}. There are several differences between the proof for the matrix $H$ and for the matrix $PHP$.  First, since the matrix $P$ is deterministic, the requirement that each $d$-cell must be crossed twice is only valid for crossings associated with entries of $H$, resulting in a weaker constraint on the set of admissible words. Second, in the F\"uredi-Koml\'os-type bound on the number of equivalence classes of words (see Lemma \ref{lem:estimation_for_trivial_ev_FK_bound}) we only obtain a rough bound, by showing that the words in the FK parsing contain at least $s-d$ crossings, instead of the $2(s-d)$ crossings proved in Lemma \ref{lem:num_of_words_in_FK_parsing}; we note that this bound may be improved using a more refined analysis, but it is sufficient for our purposes and has the advantage of having a relatively concise proof. Finally, the diagonal entries of the matrix $P$ are non-zero, which leads to a slightly larger class of words. The losses resulting from the weaker restriction on the set of admissible words and the weaker lower bound on the number of crossings in words of the FK parsing are compensated by the additional factor $n^{-1}$ associated with each entry of $P$, which ultimately allows us to improve the final estimate by the required factor $(nq)^{-1/2}$.

The proof starts with the following observation. Let $Q:=nP=\mathbb{A}+dI$ and $B:=A-\mathbb{E}[A]$. Then, analogously to \eqref{eq:semicircle_for_H_1}, we have 
\begin{equation}
\mathbb{E}\left[\int_{\mathbb{R}}x^{k}L_{PHP}(\mathrm{d}x)\right] 
 \leq\frac{1}{n^{d-1}(np)^{k/2}n^{k}} \sum_{\sigma_{1},\ldots,\sigma_{2k}\in X_{+}^{d-1}}|\mathbb{E}[B_{\sigma_{1}\sigma_{2}}Q_{\sigma_{2}\sigma_{3}}\ldots B_{\sigma_{2k-1}\sigma_{2k}}Q_{\sigma_{2k}\sigma_{1}}]|\label{eq:trivial_eigenvalues_estimation_1}
\end{equation}

We wish to rewrite the last sum using closed word, similarly to \eqref{eq:semicircle_H_6}. In order to do that a new definition for words is needed. 

\begin{defn}[$2$-words]
	An $(n,d)$-word $w$ of type $2$ (or shortly a \emph{$2$-word}) is a finite sequence $\sigma_{1}\ldots\sigma_{2k}\sigma_{2k+1}$ of letters at least three letters long such that:
	\begin{itemize}
		\item $\sigma_{2i-1}\cup\sigma_{2i}$ is a $d$-cell for every $1\leq i\leq k$. 
		\item For every $1\leq i\leq k$ either $\sigma_{2i}\cup\sigma_{2i+1}$
			is a $d$-cell, or $\sigma_{2i}=\sigma_{2i+1}$. 
	\end{itemize}
The \emph{length} of $w=\sigma_{1}\ldots\sigma_{2k+1}$ is defined to be $2k+1$. A $2$-word is called \emph{closed} if its first and last letters are the same. Two $2$-words $w=\sigma_{1}\ldots\sigma_{2k+1}$ and $w'=\sigma'_{1}\ldots\sigma'_{2k+1}$
are called \emph{equivalent} if there exists a permutation $\pi$ on $X^{0}=V$ such that $\pi(\sigma_{i})=\sigma'_{i}$ for every $1\leq i\leq 2k+1$ (see Definition \ref{def:letter_words_and_equivalence} for the meaning of $\pi(\sigma_i)$). 
\end{defn}

\begin{defn}[Support of $2$-words]
	For a $2$-word $w=\sigma_{1}\ldots\sigma_{2k+1}$ we define its \emph{support} $\mathrm{supp}_{0}(w)=\sigma_{1}\cup\sigma_{2}\cup\ldots\cup\sigma_{2k+1}$, 
its \emph{$d$-cell support}
\begin{equation*}
\mathrm{supp}_{d}(w)=\{ \sigma_{i}\cup\sigma_{i+1}\,:\,1\leq i\leq2k \text{ such that } |\sigma_i \cup \sigma_{i+1}| = d+1\},
\end{equation*}
and its \emph{odd-crossing $d$-cell support} $\mathrm{supp}_{d}^{\mathrm{odd}}(w)=\{ \sigma_{2j-1}\cup\sigma_{2j}\,:\,1\leq j\leq k\}$.
\end{defn}

\begin{defn}[The graph of a $2$-word] Given a $2$-word $w=\sigma_{1}\ldots\sigma_{2k+1}$ we define $G_{w}=(V_{w},E_{w})$ to be the graph with vertex set $V_{w}=\{ \sigma_{i}\,:\,1\leq i\leq2k+1\}$ and edge set $E_{w}=\{ \{ \sigma_{i},\sigma_{i+1}\} \,:\,1\leq i\leq2k\}$; note that we may have $\sigma_i = \sigma_{i+1}$, in which case $\{\sigma_i, \sigma_{i+1}\}$ gives rise to a loop in $E_w$. For an edge $e\in E_{w}$ we define $N_{w}(e)=|\{ 1\leq i\leq2k+1\,:\,\{ \sigma_{i},\sigma_{i+1}\} =e\}|$ to be the number of times the edge $e$ is crossed along the path generated by $w$ in the graph $G_{w}$, and let $N_{w}^\mathrm{odd}(e)=|\{ 1\leq j\leq k\,:\,\{ \sigma_{2j-1},\sigma_{2j}\} =e\}|$ be the number of crossings made in odd steps. As in the case of words we define the $i$-th crossing time of an edge and of a $d$-cell. In addition, the edges of the graph $G_{w}$ that are not loops can be divided into different classes according to the $d$-cell generated by the two $(d-1)$-cells which form its vertices. For a $d$-cell $\tau$ we define $E_{w}(\tau)=\{ \{ \sigma,\sigma'\} \in E_{w}\,:\,\sigma\cup\sigma'=\tau\} $ and $N_{w}(\tau)=\sum_{e\in E_{w}(\tau)}N_{w}(e)$, and let  $N_{w}^\mathrm{odd}(\tau)=\sum_{e\in E_{w}(\tau)}N_{w}^\mathrm{odd}(e)$. 
\end{defn}

Using the above definitions and going back to \eqref{eq:trivial_eigenvalues_estimation_1} we see that each term in the sum can be associated with a list of letters $w=\sigma_{1}\sigma_{2}\ldots\sigma_{2k}\sigma_{2k+1}$ with $\sigma_{2k+1}=\sigma_{1}$. Since $B_{\sigma,\sigma'}=0$ whenever $\sigma\cup\sigma'\notin X^{d}$ and $Q_{\sigma,\sigma'}=0$ unless $\sigma\cup\sigma'\in X^{d}$ or $\sigma=\sigma'$ it follows that we can restrict the sum in (\ref{eq:trivial_eigenvalues_estimation_1}) to those $w=\sigma_{1}\ldots\sigma_{2k+1}$ which are closed $2$-words of length $2k+1$. Using the independence structure of $A$ for different
$d$-cells and the definitions of $N_{w}$ and $N_{w,\mathrm{odd}}$ we then have
\begin{align*}
\mathbb{E}\left[\int_{\mathbb{R}}x^{k}L_{PHP}(\mathrm{d}x)\right]
&\leq \frac{1}{n^{d-1}(nq)^{\frac{k}{2}}n^{k}}\sum_{\substack{w\,\,\footnotesize{\mbox{a closed \ensuremath{2}-word}}\\
\footnotesize{\mbox{of length }}k+1
}
}\left|\prod_{\tau\in X^{d}}\mathbb{E}\left[(\chi-p)^{N_{w}^\mathrm{odd}(\tau)}\right]\prod_{\sigma\in X^{d-1}}d^{N_{w}(\{ \sigma,\sigma\} )}\right|\\
&\leq \frac{d^{2k+1}}{n^{d-1}(nq)^{\frac{k}{2}}n^{k}}\sum_{\substack{w\,\,\footnotesize{\mbox{a closed \ensuremath{2}-word}}\\
\footnotesize{\mbox{of length }}k+1
}
}\left|\prod_{\tau\in X^{d}}\mathbb{E}\left[(\chi-p)^{N_{w}^\mathrm{odd}(\tau)}\right]\right|.
\end{align*}
Since $\mathbb{E}[\chi-p]=0$, it follows that we can restrict the last sum to those $2$-words for which $N_{w}^{\mathrm{odd}}(\tau)\neq1$ for every $\tau\in X^{d}$. 

Next, we rewrite the sum over such words using their equivalence classes. Denoting by $[w]$ the equivalence class of a $2$-word $w$, we have 
\begin{equation}
	\mathbb{E}\left[\int_{\mathbb{R}}x^{k}L_{PHP}(\mathrm{d}x)\right]\leq\frac{d^{2k+1}}{n^{d-1}(nq)^{\frac{k}{2}}n^{k}}\sum_{w}|[w]|\left|\prod_{\tau\in X^{d}}\mathbb{E}\left[(\chi-p)^{N_{w}^\mathrm{odd}(\tau)}\right]\right|,\label{eq:trivial_eigenvalues_estimation_2}
\end{equation}
where the sum is over representatives of the equivalence classes of
closed $2$-words of length $2k+1$ such that $N_{w}^\mathrm{odd}(\tau)\neq1$
for every $\tau\in X^{d}$. 

We further distinguish between different equivalence classes according to the number of $0$-cells in $\mathrm{supp}_{0}(w)$. Denoting by $\widehat{\mathcal{W}}_{s}^{k}=\widehat{\mathcal{W}}_{s}^{k}(n,d)$ a set of representatives for the equivalence classes of closed $2$-words of length $2k+1$ such that $N_{w}^\mathrm{odd}(\tau)\neq1$ for all $\tau\in X^{d}$ and $|\mathrm{supp}_{0}(w)|=s$, and observing that by the same argument used to prove Claim \ref{claim:B_n_s_d} the number of elements in the equivalence class of $w\in\widehat{\mathcal{W}}_{s}^{k}$ is $B_{n,s,d}\equiv\frac{n(n-1)\cdots(n-s)}{d!}$, we can rewrite (\ref{eq:trivial_eigenvalues_estimation_2}) as 
\begin{align}
\mathbb{E}\left[\int_{\mathbb{R}}x^{k}L_{PHP}(\mathrm{d}x)\right] & \leq\sum_{s\geq d}\sum_{w\in\widehat{\mathcal{W}}_{s}^{k}}\frac{d^{2k+1}n^{s}}{n^{d-1}(nq)^{\frac{k}{2}}n^{k}}\left|\prod_{\tau\in X^{d}}\mathbb{E}\left[(\chi-p)^{N_{w}^\mathrm{odd}(\tau)}\right]\right|\nonumber \\
 & \leq\sum_{s\geq d}\frac{d^{2k+1}n^{s}q^{|\mathrm{supp}_{d}^{\mathrm{odd}}(w)|}}{n^{d-1}(nq)^{\frac{k}{2}}n^{k}}\,|\widehat{\mathcal{W}}_{s}^{k}|,\label{eq:trivial_eigenvalues_estimation_3}
\end{align}
where for the second inequality we used the same argument that yields
(\ref{eq:semicircle_H_8}). 

The next lemma gives constraints on the value of $|\mathrm{supp}_{d}^{\mathrm{odd}}(w)|$ for a given $s$ as well as a bound on the values of $s$ itself. 

\begin{lem}\label{lem:Lemma_trivial_eigenvalues_estimation_1} 
We have $\frac{|\mathrm{supp}_{0}(w)|-d}{2}\leq|\mathrm{supp}_{d}^{\mathrm{odd}}(w)|\leq\lfloor \frac{k}{2}\rfloor $, and therefore $|\mathrm{supp}_{0}(w)|\leq k+d$. 
\end{lem}

\begin{proof}
	Recalling that $N_{\mathrm{odd}}(\tau)\neq1$ for all $\tau\in X^{d}$ and using the fact that $\sum_{\tau\in X^{d}}N_{w,\mathrm{odd}}(\tau)=k$ it follows that $|\mathrm{supp}_{d}^{\mathrm{odd}}(w)|\leq\lfloor \frac{k}{2}\rfloor$. As for the lower bound, note that if $\sigma_{j}$ is the first appearance of a new $0$-cell, then both $\sigma_{i-1}\sigma_{i}$ and $\sigma_{i}\sigma_{i+1}$ are crossings in which this $0$-cell appeared. Consequently, the $d$-cells crossed in odd times contain all $0$-cells in $\mathrm{supp}_{0}(w)$, and each odd crossing can reveal at most two new $0$-cells.
\end{proof}

Combining (\ref{eq:trivial_eigenvalues_estimation_3}) and Lemma \ref{lem:Lemma_trivial_eigenvalues_estimation_1} yields  
\begin{equation}
	\mathbb{E}\left[\int_{\mathbb{R}}x^{k}L_{PHP}(\mathrm{d}x)\right]\leq\frac{d^{2k+1}}{(nq)^{\frac{k}{2}}n^{k-1}}\sum_{s=d}^{k+d}(n\sqrt{q})^{s-d} |\widehat{\mathcal{W}}_{s}^{k}|.
\label{eq:trivial_eigenvalues_estimation_4}
\end{equation}
What remains is an F\"uredi-Koml\'os-type bound on $|\widehat{\mathcal{W}}_{s}^{k}|$, which is the content of the following lemma.
\begin{lem}
\label{lem:estimation_for_trivial_ev_FK_bound} For every $d\leq s\leq k+d$
\[
|\widehat{\mathcal{W}}_{s}^{k}|\leq 2d\,(4d)^{k}\,(2k+1)(\sqrt{d}(2k+1)^{3})^{2k+1-(s-d)}
\]
\end{lem}

Before proving Lemma \ref{lem:estimation_for_trivial_ev_FK_bound} we use it to conclude the proof of Proposition \ref{prop:bound_on_eigenvalues_of_PHP}. Combining (\ref{eq:trivial_eigenvalues_estimation_4}) and Lemma \ref{lem:estimation_for_trivial_ev_FK_bound} we get for even $k$ 
\begin{align*}
	\mathbb{E}\left[\int_{\mathbb{R}}x^{k}L_{PHP}(\mathrm{d}x)\right] & \leq 2d^{2}(2k+1)\frac{(4d^{3})^{k}(\sqrt{d}(2k+1)^{3})^{2k+1}}{(nq)^{\frac{k}{2}}n^{k-1}}\sum_{s=d}^{k+d}\left(\frac{n\sqrt{q}}{\sqrt{d}(2k+1)^{3}}\right)^{s-d}\\
 & \leq 2d^{2}(2k+1)\frac{(4d^{3})^{k}(\sqrt{d}(2k+1)^{3})^{2k+1}}{(nq)^{\frac{k}{2}}n^{k-1}}(k+1)\left[\left(\frac{n\sqrt{q}}{\sqrt{d}(2k+1)^{3}}\right)^{k}+1\right]\\
 & \leq2d^{5/2}(2k+1)^{5}n\left[\left(\frac{4d^{7/2}(2k+1)^{3}}{\sqrt{n}}\right)^{k}+\left(\frac{4d^{4}(2k+1)^{6}}{\sqrt{nq} n}\right)^{k}\right].
\end{align*}

It follows from Markov's inequality that for every even $k$ 
\begin{align*}
 & \mathbb{P}\left(\|PHP\| >\eta\frac{\log^{3}n}{\sqrt{n}}\right)\\
& \qquad \leq  2d^{5/2}(2k+1)^{5}\frac{n^{d}}{(d-1)!}\left[\left(\frac{4d^{7/2}(2k+1)^{3}}{\eta\log^{3}n}\right)^{k}+\left(\frac{4d^{4}(2k+1)^{6}}{\sqrt{q} \, n \, \eta \log^{3}n}\right)^{k}\right].
\end{align*}
By taking $k=k(n)$ to be the largest even integer which is smaller than $\frac{1}{2}(\eta/(4d^{7/2}e))^{1/3}\log n-1$ and assuming that $nq \geq\frac{\eta^2}{16d^{6}e^2} \frac{\log^6 n}{n}$ we obtain that 
\begin{align*}
& \mathbb{P}\left(\|PHP\| >\eta\frac{\log^{3}n}{\sqrt{n}}\right) \\
& \qquad \leq\frac{4e^{3}d^{5/2}}{(d-1)!}\exp\left(\frac{5}{3}\log\left(\frac{\eta}{4d^{7/2}e}\right)+5\log\log n-\left[\frac{1}{2}\left(\frac{\eta}{4d^{7/2}e}\right)^{1/3}-d\right]\log n\right),
\end{align*}
thus completing the proof. \hfill{}$\boxempty$

\begin{proof}[Proof of Lemma \ref{lem:estimation_for_trivial_ev_FK_bound}]
We closely follow the proof of Proposition \ref{prop:Main_lemma_for_norm_bound_on_H}.
Given $w\in\widehat{\mathcal{W}}_{s}^{k}$ we define its FK parsing $a_{w}$ by parsing $w$ in the crossing $\sigma_{i}\sigma_{i+1}$ if one of the following occurs:
\begin{itemize}
	\item $\{ \sigma_{i},\sigma_{i+1}\} $ is an old edge of $G_{w}$. 
	\item $\{ \sigma_{i},\sigma_{i+1}\} $ is a third or subsequent crossing of the edge $\{ \sigma_{i},\sigma_{i+1}\} $. 
	\item $\{ \sigma_{i},\sigma_{i+1}\} $ is a second crossing of the edge $\{ \sigma_{i},\sigma_{i+1}\} $ and there exists an edge $\{ \sigma,\sigma'\} \in E_{w}(\sigma_i\cup\sigma_{i+1})$ in the $d$-cell
as $\sigma_{i}\cup\sigma_{i+1}$ which is crossed twice in $\sigma_{1}\sigma_{2}\ldots\sigma_{i}$.
	\item $\{ \sigma_{i},\sigma_{i+1}\} $ is a loop. 
\end{itemize}

The resulting sentence is an FK sentence in the sense of Definition \ref{def:FK-words and FK-sentences}. Since the path of each word in $\widehat{\mathcal{W}}_{s}^{k}$ must contain $s-d$ crossing times in which a new $0$-cell is observed, it follows that the number of words $m_{w}$ in the FK sentence $a_{w}$ satisfies 
\[
	1\leq m_{w}\leq2k+1-(s-d).
\]

By repeating the same argument as in the proof of Proposition \ref{prop:Main_lemma_for_norm_bound_on_H} and using Lemma \ref{lem:FK_lemma_1} and Lemma \ref{lem:How_to_concatenate_a_new_FK_sequence} we obtain 
\begin{align*}
	|\widehat{\mathcal{W}}_{s}^{k}| & \leq\sum_{m=1}^{2k+1-(s-d)}\binom{2k}{m-1}\left(\frac{\sqrt{d}}{2}\right)^{m}(2\sqrt{d})^{2k+1}(2k+1)^{2(m-1)}\\
 & \leq2d(4d)^{k}(2k+1)(\sqrt{d}(2k+1)^{3})^{2k+1-(s-d)},
\end{align*}
thus completing the proof. 
\end{proof}

{\small
\bibliographystyle{alpha}
\bibliography{citations}
}

\medskip{}
\medskip{}
$ $\\
Antti Knowles \\
E-mail: knowles@math.ethz.ch\\
$ $\\
Ron Rosenthal\\
E-mail: ron.rosenthal@math.ethz.ch\\
\medskip{}

\noindent
Department Mathematik\\
ETH Zürich \\
CH-8092 Zürich \\
Switzerland\\

\end{document}